\documentclass[11pt,leqno]{article}
\usepackage{geometry}
\usepackage{amsmath,amscd,amssymb,amsfonts,amsthm}
\usepackage{mathrsfs}
\usepackage[sans]{dsfont} 
\usepackage{hyperref}
\usepackage{authblk}
\numberwithin{equation}{section}       
\numberwithin{figure}{section}       

\usepackage{upgreek}
\usepackage[T1]{fontenc}
\usepackage[oldstylenums]{kpfonts} 
\usepackage{sectsty}
\sectionfont{\normalfont}

\usepackage{tikz}
\usetikzlibrary{patterns}
\usepackage{pgflibraryarrows}
\usepackage{pgflibrarysnakes}
\usetikzlibrary{calc}
\newcommand{\arcThroughThreePoints}[4][]{
\coordinate (middle1) at ($(#2)!.5!(#3)$);
\coordinate (middle2) at ($(#3)!.5!(#4)$);
\coordinate (aux1) at ($(middle1)!1!90:(#3)$);
\coordinate (aux2) at ($(middle2)!1!90:(#4)$);

\coordinate (center) at ($(intersection of middle1--aux1 and middle2--aux2)$);
\draw[#1] 
 let \p1=($(#2)-(center)$),
      \p2=($(#4)-(center)$),
      \n0={veclen(\p1)},       
      \n1={atan2(\x1,\y1)}, 
      \n2={atan2(\x2,\y2)},
      \n3={\n2>\n1?\n2:\n2+360}
    in (#2) arc(\n1:\n3:\n0);
\draw[ultra thick] 
 let \p1=($(#3)-(center)$),
      \p2=($(#4)-(center)$),
      \n0={veclen(\p1)},       
      \n1={atan2(\x1,\y1)}, 
      \n2={atan2(\x2,\y2)},
      \n3={\n2>\n1?\n2:\n2+360}
    in (#3) arc(\n1:\n3:\n0);
}

\newtheorem{theo}{Theorem}
\newtheorem{prop}{Proposition}[section]

\newtheorem{coro}[prop]{Corollary}
\newtheorem{lemm}[prop]{Lemma}
\newtheorem{Thm}[prop]{Theorem}
\newtheorem{theoalph}{Theorem}

\newtheorem{propalph}[theoalph]{Proposition}

\newtheorem{rema}[prop]{Remark}

\newcommand{\C}{\mathbb{C}}
\renewcommand{\D}{\mathbb{D}}

\renewcommand{\H}{\mathbb{H}}

\newcommand{\Q}{\mathbb{Q}}
\newcommand{\R}{\mathbb{R}}

\newcommand{\Z}{\mathbb{Z}}

\newcommand{\cO}{\mathcal{O}}

\newcommand{\fM}{\mathfrak{M}}

\newcommand{\ttheta}{\widetilde{\theta}}

\newcommand{\dd}{\hspace{1pt}\operatorname{d}\hspace{-1pt}}

\DeclareMathOperator{\hyp}{hyp} 
\DeclareMathOperator{\SL}{SL}

\newcommand{\CC}{\overline{\C}}

\newcommand{\Fh}{\operatorname{h}_{\operatorname{F}}}
\newcommand{\Fmuess}{\upmu_{\operatorname{F}}^{\operatorname{ess}}}

\newcommand{\LL}{\mathcal{L}}
\newcommand{\Lbar}{\overline{\mathcal{L}}}
\newcommand{\LLL}{\overline{\LL}}

\newcommand{\height}[1]{\operatorname{h}_{#1}}
\newcommand{\muess}[1]{\upmu_{#1}^{\operatorname{ess}}}
\newcommand{\muabs}[1]{\upmu_{#1}^{\operatorname{abs}}}

\newcommand{\ghyp}{g_{\hyp}}

\DeclareMathOperator{\Gal}{\textsc{Gal}}

\newcommand{\whf}{j_\D}
\newcommand{\whh}{\widehat{h}}
\newcommand{\whj}{\widehat{j}}

\newcommand{\sg}{\sigma}
\newcommand{\east}{E_2^*}

\newcommand{\gH}{g_{\infty}}

\newcommand{\h}{\H}
\newcommand{\eps}{\varepsilon}
\newcommand{\ph}{\varphi}
\newcommand{\ce}{\C}
\newcommand{\q}{\Q}
\newcommand{\erre}{\R}
\newcommand{\z}{\Z}

\newcommand{\dv}{\textrm{div}}
\newcommand{\Spec}{\textrm{Spec}}
\newcommand{\p}{\mathbb{P}^1}
\DeclareMathOperator{\Pet}{Pet}
\DeclareMathOperator{\can}{can}
\newcommand{\npet}[1]{\Vert #1 \Vert_{\Pet}}
\newcommand{\npetv}[1]{\Vert #1 \Vert_{\Pet, v}}

\newcommand{\ncanv}[1]{\Vert #1 \Vert_{\can, v}}
\newcommand{\X}{\mathcal{X}}

\title{On the essential minimum of Faltings' height}

\author{
Jos\'e Ignacio Burgos Gil,\footnote{
Burgos was partially supported by the MINECO research projects MTM2013-42135-P and ICMAT Severo Ochoa SEV-2015-0554.
} \,
Ricardo Menares
\&
Juan Rivera-Letelier\footnote{
Rivera-Letelier was partially supported by FONDECYT grant 1141091.
}
}

\newcommand{\Addresses}{{
\bigskip
\footnotesize
Jos\'e Ignacio Burgos Gil.
Instituto de Ciencias Matem\'aticas (CSIC-UAM-UCM-UCM3).
Calle Nicol\'as Ca\-bre\-ra~15, Campus UAB, Cantoblanco, 28049 Madrid,
Spain
\\
\texttt{burgos@icmat.es}
\\
\url{http://www.icmat.es/miembros/burgos}

\medskip
Ricardo Menares.
Instituto de Matem\'aticas, Pontificia Universidad Cat\'olica de Valpara\'iso, Blanco Viel 596, Cerro Bar\'on, Valpara\'iso, Chile.
\\
\texttt{ricardo.menares@pucv.cl}
\\
\url{http://ima.ucv.cl/academico/ricardo-menares/}

\medskip
Juan Rivera-Letelier.
Department of Mathematics, University of Rochester. Hylan Building, Rochester, NY 14627, U.S.A.
\\
\texttt{riveraletelier@gmail.com}
\\
\url{http://rivera-letelier.org}
}}

\begin{document}

\maketitle

\abstract{ We  study the essential minimum of the (stable) Faltings height on the moduli space of elliptic curves. We prove that, in contrast to the Weil height on a projective space and the N{\'e}ron-Tate height of an abelian variety, Faltings' height takes at least two values that are smaller than its essential minimum.
We also provide upper and lower bounds for this quantity that allow us to compute it  up to five decimal places. In addition, we give numerical evidence that there are at least four isolated values before the essential minimum.

One of the main ingredients in our analysis is a good approximation of the hyperbolic Green function associated to the cusp of the modular curve of level one.
To establish this approximation, we make an intensive use of distortion theorems for univalent functions.

Our results have been motivated and guided by numerical experiments that are described in detail in the companion files.
}

\section{Introduction}
In this article, we study the essential minimum of the (stable) Faltings height on the moduli space of elliptic curves.
Our main result is that, in contrast to the Weil height on a projective space and the N{\'e}ron-Tate height of an abelian variety, Faltings' height takes at least two values that are smaller than its essential minimum.
Actually, our numerical experiments suggest that there are at least four such values: The first one is taken at the class of elliptic curves with $j$\nobreakdash-invariant zero and the other three are taken at certain classes of elliptic curves whose $j$\nobreakdash-invariant is a root of unity.
We give a rigorous proof that there can be at most six classes of elliptic curves whose $j$\nobreakdash-invariant is a root of unity and whose Faltings' height is smaller than the essential minimum.

We now proceed to describe our results more precisely.
To recall the definition of Faltings' height, let $\h \coloneqq \left\{ \tau \in \ce : \Im(\tau) > 0 \right\}$ be the upper half-plane, and let~$\Delta \colon \H \to \C$ be the modular discriminant, normalized so that the product formula reads
$$ \Delta(\tau) \coloneqq q \prod_{n=1}^\infty (1-q^n)^{24}, \quad q=e^{2\pi i \tau}. $$
Furthermore, consider the hyperbolic Green function $g_\infty \colon \h \rightarrow \erre$, defined by
$$ g_\infty(\tau) \coloneqq  -\log  \left(\left(4 \pi \Im (\tau)\right) ^6 |\Delta (\tau)| \right) . $$
This function is invariant under the action of $\SL_2(\z)$ on $\h$.
Since $\Delta$ does not vanish on $\h$, the function~$g_\infty$ is finite and continuous.

Faltings' height is a numerical invariant attached to each abelian variety defined over a number field.
To define it in the case of an elliptic curve~$E$ defined over a number field~$K$, denote by~$\Delta_{E/K}$ the minimal discriminant of $E/K$.
Furthermore, for a given embedding $\sg \colon K \rightarrow \ce$, choose~$\tau_\sg \in \h$ such that
$$E_{\sg}(\C) \cong \ce /(\z + \tau_\sg \z),$$
where $E_{\sg}$ is the elliptic curve over $\C$ obtained from $E$ by base
change using $\sg$. 
Then the \emph{Faltings' height~$\Fh(E/K)$ of~$E/K$} can be defined as
\begin{equation}\label{defFaltings}
\Fh(E/K)
\coloneqq
\frac{1}{12[K:\q]} \left( \log \left| N_{K/\q} \left(\Delta_{E/K} \right) \right| + \sum_{\sg \colon K \rightarrow \ce  }^{ }  g_\infty(\tau_\sg) \right),
\end{equation}
see for example~\cite[Proposition~$1.1$]{Silverman}.\footnote{We warn the reader that there are different normalizations of~$\Fh$ in the literature, any two of them differing by an additive constant. In order to compare results by diverse authors, we have preferred a normalization different from the one in \emph{loc. cit.}}
If~$L/K$ is a finite extension such that $E_L\coloneqq E\otimes L$ is semistable, then it is not hard to check that $\Fh(E_L /L) \leq \Fh(E/K)$.
Moreover, the quantity $\Fh(E_L /L)$ does not depend on the choice of $L$.
In other words, on a given $\overline{\Q}$-isomorphism class of elliptic curves, Faltings' height attains its minimum at a semistable representative and its value does not depend on the choice of such semistable elliptic curve.

\emph{Faltings' height function~$\Fh$} is the induced function
$$ \Fh \colon \overline{\Q} \rightarrow \R $$
that to a given algebraic number $\alpha \in \overline{\Q}$ attaches the real number~$\Fh(\alpha) \coloneqq \Fh(E_\alpha/L)$, where $L$ is a number field containing~$\alpha$ and~$E_\alpha$ is a semistable elliptic curve defined over~$L$ with $j$-invariant equal to~$\alpha$.\footnote{In the literature this function is also called the ``stable Faltings height'' function.}
Faltings showed that the function $\Fh$ behaves as a height on the moduli space of elliptic curves (\emph{e.g.}, it satisfies Northcott's property) and became a standard tool in diophantine geometry.

Our main results concern the essential minimum~$\Fmuess$ of Faltings' height function, defined by
$$ \Fmuess
\coloneqq
\inf \left\{ \theta \in \erre : \textrm{ the set } \left\{ \alpha \in \overline{\Q} : \Fh(\alpha) \leq \theta \right\} \textrm{ is infinite} \right\}.$$ 
Note that the set
\begin{equation}
  \label{e:spectrum}
\left\{ \Fh(\alpha) : \alpha \in \overline{\Q} \right\} \setminus \left[ \Fmuess, \infty \right)
\end{equation}
is either finite, or formed by an increasing sequence converging to~$\Fmuess$.

In the case of the Weil height on a projective space, the N{\'e}ron-Tate height of an abelian variety, and the canonical height of a polarized dynamical system, the set corresponding to~\eqref{e:spectrum} is empty.
Our first main result is that, in contrast, the set~\eqref{e:spectrum} contains at least two elements: The first minimum of~$\Fh$ is~$\Fh(0)$, and the second~$\Fh(1)$.

\begin{theo}
\label{t:minima}
We have
\begin{equation}\label{cadena}
\Fh(0) < \Fh(1) < \Fmuess,
\end{equation}
and there exists $\kappa>0$ such that for every algebraic number~$\alpha \neq 0,1$ we have $\Fh(\alpha)\geq \Fh(1)+\kappa$.
\end{theo}

Our numerical experiments suggest that in fact the set~\eqref{e:spectrum} contains at least four elements, and that its smallest elements, besides~$\Fh(0)$ and~$\Fh(1)$, are given by the values of~$\Fh$ taken at the primitive roots of unity of orders~$6$ and~$10$.
See the summary below, Section~\ref{sec:numer-exper}, and the companion files~\cite{BMRan} for precisions.
Our second main result is that among the values of Faltings' height taken at roots of unity, these are the only ones that could belong to~\eqref{e:spectrum}, with the possible exception of the values of~$\Fh$ at the primitive roots of unity of orders~$14$, $15$ and~$22$.

\begin{theo}
\label{t:roots of unity}
Let~$n \ge 2$ be an integer different from~$6$, $10$, $14$, $15$ and~$22$, and let~$\zeta_n$ be a primitive root of unity of order~$n$.
Then~$\Fh(\zeta_n) > \Fmuess$.
\end{theo}

The estimates used to prove Theorems~\ref{t:minima} and~\ref{t:roots of unity} easily yield the following.
\begin{coro}
\label{c:numerical}
We have
$$ 10^{-4} < \Fh(1) - \Fh(0) < \Fmuess - \Fh(0) < 2 \cdot 10^{-4}
\text{ and }
-0.748629  \leq \Fh(1) < \Fmuess \leq -0.748622, $$
and the set of values of~$\Fh$ is dense in the interval~$[-0.748622, \infty)$.

On the other hand, if~$\alpha$ is an algebraic number whose Faltings' height is less than or equal to~$\Fmuess$, then~$\alpha$ is either an algebraic integer or of degree greater than or equal to~$10520$. Moreover, if the degree of $\alpha$ is at most $10$, then $\alpha$ is an algebraic unit.
\end{coro}

We did not try to get the best possible numerical estimates from the method we are using.
We opted for weaker numerical estimates that are easier to verify.

The fact that the minimum value of~$\Fh$ is
\begin{equation}
  \label{gmin}
\Fh(0)
=
\frac{1}{12} \gH \left( e^{\pi i / 3} \right)
=
-\frac{1}{2} \cdot \log \left( \frac{3}{(2\pi)^{3}} \Gamma \left( \frac{1}{3} \right)^{6}  \right)
=
-0.748752485503338...,
\end{equation}
was observed by Deligne in~\cite[p.~$29$]{Del85}.
The inequality $\Fh(0) < \Fmuess$ has been observed independently by L{\"o}brich \cite{Lobrich}, showing that $\Fmuess - \Fh(0) \geq 4.601 \cdot 10^{-18}$.
In~\cite{Zagier93}, Zagier studied a height function for which the set
corresponding to~\eqref{e:spectrum} also contains at least two
points. In \cite{Doc01} and \cite{OtroDoc01}, Doche continued the
study of such height function and determined an  upper bound and a
computer assisted lower bound for  the corresponding essential
minimum.

The following is a summary of what we have found in our numerical experiments, which have motivated and guided our results:
\begin{itemize}
\item
\textit{\textsf{First four minima}}:
\begin{displaymath}
\Fh(0)=-0.74875248 \ldots,
\Fh(1)=-0.74862817 \ldots,
\end{displaymath}
\begin{displaymath}
\Fh(\rho)=-0.74862517 \ldots,
\Fh(\xi)=-0.74862366 \ldots,
\end{displaymath}
where $\rho$ is a primitive root of unity of order~$6$, a root of the polynomial $z^2-z+1$, and~$\xi$ is a primitive root of unity of order~$10$, a root of $z^4-z^3+z^2-z+1$.
\item
\textit{\textsf{Next known value}}: $-0.74862330 \ldots$, taken at the roots of the polynomial
$$ z^8 - 2z^7 + 2z^6 - z^5 + z^4 - z^3 + z^2 - z + 1. $$
\item
\textit{\textsf{Bounds for the essential minimum:}} $-0.74862345 \le \Fmuess \le -0.74862278$.
\item 
\textit{\textsf{Density interval}}: The values of~$\Fh$ are dense in the interval~$[-0.74862278, 
\infty)$. 
\end{itemize}
See Section~\ref{sec:numer-exper} and the companion files~\cite{BMRan} for further details.
Note in particular that our numerical experiments locate the essential minimum~$\Fmuess$ in an interval of length smaller than~$10^{-6}$.
Furthermore, $\Fh$ takes exactly four values to the left of this interval and the values of~$\Fh$ are dense to the right of this interval.

We remark that only the second chain of inequalities in Corollary~\ref{c:numerical} depends on the chosen normalization of Faltings' height.  

The set of Faltings' heights of elliptic curves that are not necessarily semistable is a dense subset of $[\Fh(0), \infty)$, so by Theorem~\ref{t:minima} it is strictly larger than the set of values of Faltings' stable height.
Actually, even the set
$$ \{ \Fh(E/K) : K \text{ is a number field and $E/K$ is an elliptic curve such that~$j(E) = 0$} \} $$
is dense in $[\Fh(0), \infty)$.
This follows from the fact that the set of prime numbers~$p$ satisfying $p \equiv 1 \mod 9$ is infinite, and from the fact that for every such~$p$ and every integer $\ell \ge 1$ there is a number field~$K$ and an elliptic curve~$E/K$ such that~$j(E) = 0$ and $h_F(E/K) = h_F(0) + \log p/(6\ell)$, see the proof of Theorem 1.3 in \cite{Lobrich}.

We now proceed to explain the main ingredients of the proofs of Theorems~\ref{t:minima} and~\ref{t:roots of unity}, and simultaneously explain how the paper is organized.
Our method is based on the interpretation of~$\Fh$ as an Arakelov-theoretic height on the modular curve of level one, induced by the line bundle~$M_{12}$ of weight~$12$ modular forms, together with the Petersson metric~$\npet{\cdot}$.
The height~$\Fh$ is computed by choosing a section of~$M_{12}$. 
In Section~\ref{Modular ingredients} we review the Arakelov-theoretic interpretation of Faltings' height.
We also collect the values at~$e^{\pi i/3}$ of the classical Eisenstein series of weight~$2$, $4$ and~$6$ and some of their derivatives, and compute~$j'''(e^{\pi i/3})$ in terms of those values.

The proof of Theorem~\ref{t:minima} is based on the following ``minimax'' procedure.
Let~$s$ be a nonzero section of~$M_{12}$.
Then, for every $\alpha\in \overline{\Q}$ outside of the  set $|\dv(s)|=\left\{ \alpha : s(\alpha) \in \{ 0, \infty\} \right\}$, we have $\Fh(\alpha) \geq \inf \left(-\log \npet{s} \right)$. Since $|\dv(s)|$ is a finite set, this yields the lower bound
\begin{equation}\label{cota general}
\Fmuess \geq \inf \left( -\log \npet{s} \right),
\end{equation}
\emph{cf}. Proposition~\ref{prop:1}.
Hence, to find a lower bound of~$\Fmuess$ one is led to search for a section~$s$ maximizing the right-hand side of~\eqref{cota general}.

For instance, the choice $s=\Delta$ yields the lower bound $\Fmuess\geq \Fh(0)$, considering that $j \left( e^{\pi i/3} \right) = 0$ is an integer and that $\gH=-\log \npet{\Delta}$ reaches its minimum at $e^{\pi i/3}$.
Since~$0$ is algebraic, a natural idea to improve this lower bound is to ``penalize'' the value~$j=0$ and look for a section of the form~$s = j^a \cdot \Delta$, for some~$a > 0$.
The technical heart of the proof of Theorem~\ref{t:minima} is to show that an appropriate choice of~$a$ yields the lower bound~$\Fh(1) \le \Fmuess$.
This is the content of the following proposition.

\begin{propalph}
\label{p:minima}
Let $\ghyp \colon \C \rightarrow \R$ be the function defined by 
\begin{equation}\label{def ghyp}
\gH = \ghyp \circ j.
\end{equation}
Then we have that $0<\partial_x \ghyp(1) <1$ and that the function~$g_1 \colon \C \setminus \{ 0 \} \to \R$ defined by
$$ g_1(\zeta) \coloneqq \ghyp(\zeta) - \partial_x \ghyp(1) \cdot \log | \zeta |, $$
attains its minimum value at, and only at, $\zeta = 1$.
\end{propalph}

See Remark~\ref{multiplicador} for an explanation of the choice~$a = \partial_x \ghyp(1)$. 
Once Proposition~\ref{p:minima} is established, an infinitesimal version of the argument above yields the strict inequality~$\Fh(1) < \Fmuess$. 
In Section~\ref{reduccion extremal} we show how to deduce Theorem~\ref{t:minima} from Proposition~\ref{p:minima}.

Our numerical experiments suggest that there are real numbers~$a_1 > 0$ and~$a_2 > 0$ such that the choice $s=j^{a_1}(j-1)^{a_2}\Delta$ leads to the more precise lower bound $\Fh(e^{\pi i/3}) \le \Fmuess$, and ultimately to the strict inequality~$\Fh(e^{\pi i/3}) < \Fmuess$.
It is possible to prove this rigorously using the methods developed in this paper, but we do not do so here in order to keep this article at a reasonable length.
We discuss further numerical experiments in Section~\ref{sec:numer-exper} and in the companion files~\cite{BMRan}.

The algorithm  described above, which is applied here to Faltings' height, is valid for a general height (\emph{c.f.} section \ref{algoritmo inferior} for a precise general formulation). In fact, this method was used in the aforementioned papers of Doche and Zagier and can be traced back to results on Mahler measures by Smyth \cite{Sm81}. See  \cite[Theorem~$3.7$]{BPS} for an application in the context of toric heights.

Another possible route to estimate~$\Fmuess$ from below is to adapt to~$\Fh$ the bounds on successive minima given by Zhang in~\cite{Zhang}.
However, this approach yields a weaker lower bound of~$\Fmuess$ than those given by Theorem~\ref{t:minima}.
See Section~\ref{modular} for further details.

One of the main ingredients in the proofs of Theorem~\ref{t:roots of unity} and Proposition~\ref{p:minima} is an approximation of~$\ghyp$ and of its first and second derivatives, on a suitable neighborhood of the unit disk.
Roughly speaking, we show that~$\ghyp$ is well approximated by the sum of a linear function and of an explicit function having a conic singularity at~$\zeta = 0$. 
The following is a sample estimate in this direction, which is used in the proof of Theorem~\ref{t:roots of unity}.
\begin{propalph}
\label{zetacontraw}
Letting
$$ \gamma_0
\coloneqq
\frac{\sqrt{3}}{\pi} \Gamma\left( \frac{1}{3} \right)^2
\text{ and }
\gamma_1
\coloneqq
3 \log(192) - 6 \log \left( \gamma_0^3 - \gamma_0^{-3} \right), $$
for every~$\zeta$ in~$S^1$ we have 
$$\left|\ghyp(\zeta) - \left(\gamma_1 - \frac{\Re(\zeta)}{13824}\right) \right| \leq 5\cdot 10^{-7}. $$
\end{propalph}

The approximation of~$\ghyp$ is achieved in two independent steps.
The first step is an approximation of the inverse function of~$j$ on a suitable neighborhood of the unit disk.
This is done in Section~\ref{Koebe}.
To this end we use the Koebe distortion theorem and several of its variants.
Loosely speaking, this result gives a quantitative estimate on how well a given univalent function (that is, an injective and holomorphic function) is approximated by its linear part at a given point.
We apply this theorem to the function induced by~$j$ on the quotient of a neighborhood of~$e^{\pi i / 3}$ in~$\H$ by the stabilizer of this point in~$\SL_2(\Z)$, which is of order three.
The computation of~$j'''(e^{\pi i/3})$, alluded above, is important in the determination of the constants in the resulting approximations.

The second step is an approximation of the function~$g_\infty$ and of its first and second derivatives, on a suitable neighborhood of~$e^{\pi i /3}$, and on a suitable coordinate.
This is done in Section~\ref{s:approximation}.

The proofs of Theorem~\ref{t:roots of unity}, Corollary~\ref{c:numerical} and Proposition~\ref{zetacontraw} are given in Section~\ref{numerical}.
After giving the proof of Proposition~\ref{zetacontraw} in Section~\ref{ss:roots of unity}, we estimate the values of~$\Fh$ at the roots of unity (Corollary~\ref{estimaciones especiales}).

Besides the approximation of~$\ghyp$ mentioned above, the main ingredient in the proof of Theorem~\ref{t:roots of unity} and Corollary~\ref{c:numerical} is a general method to estimate the essential minimum from above, which is based on the classical Fekete-Szeg{\"o} theorem and an equidistribution result from~\cite{BPRS}. See section \ref{ss:upper bound} for the precise formulation of the method.  Here, we apply it to Faltings' height, but  is also valid for other heights. 

Since on a relatively large neighborhood of the unit disk the function~$\ghyp$ is very close to a function having radial symmetry, the integral of~$\frac{1}{12} \ghyp$ against the Haar measure of~$S^1$ gives a very good upper bound of~$\Fmuess$.
However, this estimate is not sufficient for the proof of Theorem~\ref{t:roots of unity}.
We use instead a better upper bound that is obtained by integarting~$\frac{1}{12} \ghyp$ against a certain translate of the Haar measure of~$S^1$.
This upper bound and the proofs of Theorem~\ref{t:roots of unity} and Corollary~\ref{c:numerical} are given in Section~\ref{ss:upper bound}.

The proof of Proposition~\ref{p:minima} is given in Section~\ref{key}.
The main part of the proof is divided in three cases, according to the proximity of~$\zeta$ to the unit disk.
By far, the most difficult case is the case where~$\zeta$ is close the unit disk.
To deal with this case we establish some convexity properties of~$\ghyp$ in this region, using the results of Sections~\ref{Koebe} and~\ref{s:approximation}.

Finally, in Section~\ref{sec:numer-exper} we discuss numerical experiments around the determination of further isolated values of $\Fh$ and  lower bounds of $\Fmuess$.  See also \cite{BMRan} for a detailed presentation of these
experiments.

\subsection*{Acknowledgements}
We thank Yuri Bilu, Gerard Freixas-i-Montplet, Sebasti\'an Herrero and Mart\'in Sombra for useful comments and references. We also thank Philipp Habegger for pointing us to \cite{Doc01}, after the first version of this paper was completed. Numerical experiments were made in PARI and SAGE.
 
 \section{Modular ingredients}\label{Modular ingredients}

\subsection{Arakelov-theoretic interpretation of Faltings' height}\label{modular}

Let $\X\coloneqq \p_{\z}$ and consider the section $s_\infty \colon \Spec (\z) \rightarrow \X$ given by $[1:0]$. We denote by~$D_\infty$ the divisor induced by this section and consider the line bundle $\LL\coloneqq O_{\X}(D_\infty)$.
On the other hand, consider the modular curve $X\coloneqq \left(
  \SL_2(\z) \backslash \h \right) \cup \{\infty\}$ and the
$j$\nobreakdash-invariant $j \colon \H \to \C$, normalized by 
$$ j(\tau) 
\coloneqq
\frac{1}{q} + 744 + \cdots, \quad q=e^{2\pi i \tau}. $$
Every elliptic curve $E$ over $\C$ has a Weierstrass equation 
  of the form 
  \begin{equation}\label{eq:18}
    y^2= x^3-27c_{4} x-54 c_{6}
  \end{equation}
with the notation of \cite[pp.~46--48]{Silver}.
Then $c_{4}$ can be seen as a modular form of weight~$4$. The modular
discriminant $\Delta $ is a modular form of weight~$12$ and we have
the relation 
\begin{equation}
  \label{eq:4}
  j=c_{4}^3/\Delta.
\end{equation} 

There is a holomorphic bijection
$$ \iota \colon X \rightarrow \X(\ce)=\p(\ce) $$
given by $\iota (\tau )=[c_{4}(\tau )^{3}:\Delta (\tau )]$ and~$\iota(\infty) = [1 : 0]$.
This
bijection identifies $j$ with the absolute coordinate of $\p(\C)$.
Moreover, this choice of coordinates 
gives an isomorphism between the line bundle~$\LL(\ce)$ and the line
bundle $M_{12}\left(\SL_2(\z)\right) \rightarrow X$ of weight~$12$
modular forms of level one, that  identifies $\Delta $ with a canonical section of the former. 
Indeed, at the level of  global sections, we have an isomorphism 
\begin{displaymath}
w\colon H^0 \left( \X(\ce), \LL(\C) \right) =\{f\mid
\dv(f)+\infty\ge 0\}\longrightarrow M_{12}\left(\SL_2(\z)\right),
\end{displaymath}
given by
\begin{equation} \label{modforms}
w(f)\coloneqq  \Delta \cdot \iota^* f.
\end{equation}

We recall that~$M_{12}$ carries the \emph{Petersson metric}, defined
for a section~$g$ in~$M_{12}(\SL_2(\Z))$ by  
$$ \npet{g} (\tau)
\coloneqq
\left(4 \pi \Im(\tau)\right)^6 |g(\tau)|. $$
We endow~$\LL(\C)$ with the metric for which~\eqref{modforms} becomes
an isometry, which we also denote by~$\npet{\cdot}$.
Let $K$ be a number field and denote by $K^0$
(resp. $K^\infty$) the set of non-archimedean (resp. archimedean)
places of $K$. For~$v$ in~$K^0$  we denote by $\ncanv{\cdot}$  the
canonical metric on $\LL \otimes K_v$. It is defined as follows.
Let $\zeta=[\zeta_0: \zeta_1]$ be the standard homogeneous
coordinates of $\p$.  Any nonzero section~$s$ of~$\LL$ can be identified canonically with
a linear form $\ell_{s}(\zeta_0,\zeta_1)$.  If $\zeta =[\zeta_0:
\zeta_1]\notin \dv (s)$,  then
$$\ncanv{s(\zeta)} \coloneqq \frac{|\ell_{s}(\zeta_0,
  \zeta_1)|_v}{\max\{ |\zeta_0|_v, |\zeta_1|_v\}}.$$
On every place~$v$ in~$K^\infty$ we denote by $\npetv{\cdot}$ the
Petersson metric on $\LL \otimes K_v$. Putting together all the
metrics, $\LL\otimes K$
becomes a metrized line bundle that we denote $\LLL$. 

Although the metrized line bundle~$\LLL$ is singular at~$[1:0]$, it
does induce a height function~$\height{\LLL}$ which is defined at
every point~$\zeta$ in~$\X(\overline{\q})$ different from~$[1 : 0]$. 
To define~$\height{\LLL}(\zeta)$, let $K$ be a number field that
contains $\zeta$.
Then, the height~$\height{\LLL}(\zeta)$ of~$\zeta$ is defined as
\begin{equation}\label{eq:9}
  \height{\LLL} (\zeta) \coloneqq \frac{ \widehat{\deg}(\Lbar|_{D_\zeta})}{[K:\q]}=
  \frac{1}{[K:\q]}\left( \sum_{v \in K^0 }-\log \ncanv{s(\zeta)} + 
    \sum_{v \in K^\infty } -\log \npetv{s(\zeta)}  \right). 
\end{equation}

\begin{lemm}\label{magia}
Let $E/K$ be an elliptic curve and let $L/K$ be
a finite extension such that $E_L\coloneqq E\otimes L$ is
semistable. Then, $\Fh(E_L/L) = \frac{1}{12} \height{\LLL} (j(E))$. 
\end{lemm}

\begin{proof}
Let $s_{\Delta }$ be  the section
  corresponding to $\Delta$ through \eqref{modforms}. The
  corresponding linear form in the homogeneous coordinates $[c_{4}^{3}:\Delta ]$ is
  again $\Delta $. Therefore
  \begin{equation}
    \label{eq:17}
    -\log \ncanv{s_{\Delta }}= -\log \frac{|\Delta |_v}{\max\{
      |c_{4}^{3}|_v, |\Delta|_{v} \}}=\log^+|j|_v, \quad \textrm{ for each } v \in L^0. 
  \end{equation}
By the independence of  $\height{\LLL} (\zeta)$ on
  the choice of $K$, the fact that \eqref{modforms} is an isometry and equation \eqref{eq:17},  the assertion boils
  down to the equality
\begin{equation}\label{magic}
-\log|\Delta_{E_L/L}|_v= \log^+|j(E)|_v, \quad \text{for each } v \in L^0.
\end{equation}

For any place $v\in L^{0}$, we choose a minimal equation for the place
$v$, having  an associated quantity $c_4(E,v)$ as in
\cite[p.~46]{Silver}.
Then
\begin{displaymath}
  |j(E)|_v=\frac{|c_4(E,v)^3|_v}{|\Delta_{E/L}|_v}.
\end{displaymath}

Assume $|j(E)|_v>1$. By hypothesis, $E$ has split multiplicative reduction. Hence, 
$|c_4(E,v)|_v=1$ (\emph{cf}. \cite[Proposition~III.1.4]{Silver}). The above relation
implies \eqref{magic}. 

Assume that $|j(E)|_v\leq 1$. Then, $E$ has good reduction at
$v$. Since $\Delta_{E/L}$ is minimal, we have that
$|\Delta_{E/L}|_v=1$ and  both sides of \eqref{magic} are zero.
\end{proof}

Now we compare the lower bounds of~$\Fmuess$ in Theorem~\ref{t:minima} with that obtained by Zhang's bounds on successive minima.
Since the Petersson metric is singular, Theorem 5.2 in \cite{Zhang} does not apply directly to our situation.
We use instead the generalization by Bost and Freixas-i-Montplet~\cite[Theorem~3.5]{Bost-Freixas}.
To state the lower bound, denote by~$\height{\LLL}(\X)$ the height of~$\X$ with respect to~$\LLL$, by~$\muess{\LLL}$ its essential minimum, and by $\zeta$ the Riemann zeta function.
Combined with the computation of~$\height{\LLL}(\X)$ in~\cite[Theorem~6.1]{Kuhn}, the lower bound reads
\begin{equation}\label{cota de Zhang}
\Fmuess
=
\frac{1}{12} \muess{\LLL}
\ge
\frac{1}{12}\cdot \frac{ \height{\LLL}(\X)}{2}
=
6 \left(\frac{1}{2}\zeta(-1) + \zeta'(-1) \right)
=
-1.2425268622 ...,
\end{equation}
which is weaker than the lower bound in Corollary~\ref{c:numerical}, and cannot be used to deduce that $\Fh(0) < \Fmuess$.
Actually, these numerical estimates together with Corollary~\ref{c:numerical} imply the following.
\begin{coro}
\label{no Zhang}
Denoting by~$\muabs{\LLL}$ the infimum of~$\height{\LLL}$ on~$\overline{\Q}$, we have $\frac{1}{2} \height{\LLL}(\X)  < \muabs{\LLL} < \muess{\LLL}$.
\end{coro}

\subsection{Lower bounds through real sections}\label{algoritmo inferior}

We consider the graded semigroup
\begin{displaymath}
  S_{\z}=\coprod_{n\ge 0}\Gamma (\X,\LL^{\otimes n})\setminus\{0\}
\end{displaymath}
with the tensor product as operation. We denote by $S_{\erre}$ the
corresponding semigroup with real coefficients. That is, any element
of $s\in S_{\erre}$, called a real global section, can be represented
(non-uniquely) as 
\begin{equation}\label{eq:3}
  s=s_{1}^{\otimes a_{1}}\otimes\dots\otimes s_{\ell}^{\otimes a_{\ell}},\quad
  s_{1},\dots,s_{\ell} \in S_\Z, a_{1},\dots,a_{\ell} >0.
\end{equation}
The \emph{support} of the divisor of~$s$ is the set
\begin{displaymath}
  |\dv(s)|\coloneqq\bigcup_{k}|\dv(s_{k})|\subset \X(\overline{\q}),
\end{displaymath}
and its \emph{weight}
\begin{displaymath}
\frac{1}{12} (a_{1}\deg(s_{1}) + \dots + a_{\ell}\deg(s_{\ell}));
\end{displaymath}
both are independent of the representation~\eqref{eq:3}.
We denote by~$S_{\erre,1}$ the space of real global sections
of weight one.
Any real global section $s\in S_{\erre}$ defines a \emph{Green function}
\begin{displaymath}
  \begin{matrix}
    g_{s}\colon & \X(\ce)&\longrightarrow & \erre\cup \{\infty\}\\
    &x&\longmapsto & -\log\npet{s(x)},
  \end{matrix}
\end{displaymath}
where
\begin{displaymath}
  \npet{s(x)}\coloneqq\prod_{i=1}^{\ell}\npet{s_{i}(x)}^{a_{i}}.
\end{displaymath}

The following is our main source of lower bounds of~$\Fmuess$.

\begin{prop}\label{prop:1}
  Let $s\in S_{\erre,1}$ be a real global section of weight one and $x\in
  \X(\ce)\setminus |\dv(s)|$ an algebraic point not belonging to the
  support of the divisor of $s$.
Then
  \begin{displaymath}
    \Fh(x)=\frac{1}{12} \height{\LLL}(x)\ge \inf_{y\in
      \X(\ce)}g_{s}(y)=-\log \sup _{y\in \X(\ce)}
    \npet{s(y)}. 
  \end{displaymath}
In particular
\begin{displaymath}
  \Fmuess \ge \inf_{y\in \X(\ce)}g_{s}(y).
\end{displaymath}
\end{prop}
\begin{proof}
  Choose a representation of $s$ as in~\eqref{eq:3}, and put
  $K \coloneqq \q(x)$. Let $\Sigma $ be the set of embeddings of $K$
  in $\ce$. Then by~\eqref{eq:9}
  \begin{displaymath}
    \Fh(x)=\sum_{i=1}^{k}a_{i}\frac{1}{[K:\q]}\left( \sum_{v \in K^0
    }-\log\ncanv{s_{i}(x)} +  
    \sum_{v \in K^\infty } -\log \npetv{s_{i}(x)} \right).
  \end{displaymath}
  Since the sections $s_{i}$ are global sections over the integer
  model $\X$, by the definition of the canonical metric be obtain that 
  $\ncanv{s_{i}(x)}\le 1$. Therefore 
  \begin{equation*}
    \begin{split}
  \Fh(x) & \ge \sum_{i=1}^{k}a_{i}\frac{1}{[K:\q]}
    \sum_{\sigma \in \Sigma } -\log \npet{s_{i}(\sigma (x))} \\
    &= \frac{1}{[K:\q]}
    \sum_{\sigma \in \Sigma } g_{s}(\sigma (x))  \\
    &\ge \inf_{y\in \X(\ce)}g_{s}(y).       
    \end{split}
  \end{equation*}
  The second statement follows directly from the first.
\end{proof}

\subsection{Review of low weight Eisenstein series}\label{eisenstein}

Here, we recall the definition and special values of some classical Eisenstein series.
Given $s \ge 0$, define $\sg_s(n)=\sum_{d |n, d \ge 1 }^{} d^s$. For $\tau\in \h$ we put $q=e^{2\pi i \tau}$. Let
 $$E_2(\tau)\coloneqq 1 - 24 \sum_{n = 1}^{\infty} \sg_1(n)q^n, \quad E_4(\tau)\coloneqq 1+240\sum_{n = 1}^{\infty} \sg_3(n)q^n, \quad E_6(\tau)\coloneqq 1-504\sum_{n = 1}^{\infty} \sg_5(n)q^n.$$

\noindent We also define $$\east(\tau)\coloneqq E_2(\tau) - \frac{3}{\pi \Im(\tau)}.$$ 

The functions~$E_4 $ and $E_6$ are modular forms of level one and weight $4$ and $6$, respectively. The function $\east$ satisfies the relations $$\east\left(-\frac{1}{\tau}\right)=\tau^2\east(\tau), \quad \east(\tau+1)=\east(\tau), \textrm{ for all } \tau \in \h,$$  but  it is not holomorphic. On the other hand, $E_2$ is holomorphic (even at infinity) but it is not a classical modular form.

Ramanujan's identities, see \emph{e.g.} \cite[Theorem~X.5.3]{Lang76}, imply the following relations 
\begin{equation}\label{Ramanujan}
 E_2'=\frac{\pi i}{6}(E_2^2-E_4) , \quad E_4'=\frac{2 \pi i}{3} (E_2E_4-E_6).
\end{equation}

\begin{lemm} \label{Edos}
Letting $\rho \coloneqq e^{\pi i /3}$, we have
\begin{enumerate}\label{valores especiales}
 \item \label{primero}
$$ E_2(\rho)=\frac{2\sqrt{3}}{\pi}, \quad E_2'(\rho)=\frac{2i}{\pi} , \quad E_2''(\rho)=-\frac{4}{\sqrt{3}\pi}-\frac{\pi^2}{9}E_6(\rho). $$
 \item \label{eseis} $E_6(\rho)=\frac{3^3}{2^9}\cdot \frac{\Gamma(1/3)^{18}}{\pi^{12}}.$
 \item \label{jjota} $j'''(\rho)= -i\pi^3\cdot 2^{10} \cdot 3 \cdot E_6(\rho).$
\item \label{ddelta} $\Delta(\rho)=-\frac{3^3}{2^{24}}\cdot \frac{\Gamma(1/3)^{36}}{\pi^{24}}.$
\end{enumerate}
\end{lemm}

\begin{proof} 
A proof of statements~\ref{eseis} and~\ref{jjota} can be found in  \cite[p.~777]{Wustholz14}.
We proceed to justify statement~\ref{primero}. Since $\east $ is weakly modular of weight two, and $\rho$ is fixed by $\tau\mapsto \frac{1}{1-\tau}$, we have that $\east(\rho)=0$, implying $E_2(\rho)= \frac{2\sqrt{3}}{\pi}$.  Similarly, using that $E_4$ is modular of weight four, we have that $E_4(\rho)=0$. Then, using~\eqref{Ramanujan}, we obtain $$E_2'(\rho) =  \frac{\pi i}{6}E_2^2(\rho) =  \frac{\pi i}{6}\left(\frac{2\sqrt{3}}{\pi}\right)^2=\frac{2i}{\pi}.$$  

Using~\eqref{Ramanujan} again, we have that  $E_2'' = \frac{\pi i}{6}(2E_2\cdot E_2' -E_4')$ and $E_4'(\rho)=-\frac{2\pi i}{3}E_6(\rho)$. Then, 

$$E_2''(\rho)=\frac{\pi i }{6} \left( 2 \cdot \frac{2\sqrt{3}}{\pi} \cdot \frac{2i}{\pi} +  \frac{2\pi i}{3}E_6(\rho)\right)=-\frac{4}{\sqrt{3}\pi} -\frac{\pi^2}{9}E_6(\rho),$$

\noindent proving claim~\ref{primero}. Finally, statement~\ref{ddelta} follows from the other statements and the identity $\Delta = \frac{1}{1728}(E_4^3-E_6^2)$ \cite[p.~9 and X.\S4, Theorem~4.1]{Lang76}.
\end{proof}

We record here a result of Masser on the zeroes and real values of~$\east$, which is shown in the proof of~\cite[Lemma~3.2]{Masser75}.
\begin{lemm}
\label{realidad Eisenstein}
The function $\east$ vanishes at, and only at, the $\SL_2(\Z)$-orbits of $i$ and $\rho$. Moreover, we have that
 $\Im \left(\east(z)\right)=0$ if and only if $\Re (z) \in \frac{1}{2}\Z$.
\end{lemm}

We denote by $\partial$ the holomorphic derivative. That is, for a given complex variable $\tau=x+iy \in \C$, we have $\partial=\frac{1}{2}(\partial_x -i \partial_y)$. 

\begin{lemm}
   The following identities hold
\begin{eqnarray}
\gH(\tau) & = & 2\pi \Im(\tau) - 6 \log(\Im(\tau)) - 6 \log(4\pi)
- 24 \sum_{r = 1}^{\infty} \log|1 - q(\tau)^r|.  \label{eq:6} \\
\partial \gH & = &-\pi i  \east. \label{integral}
 \end{eqnarray}

\end{lemm}

\begin{proof} Equation~\eqref{eq:6} is a direct consequence of the product formula for the modular discriminant. We then deduce

\begin{eqnarray}
 \partial \gH(\tau) &=&\frac{3 i}{\Im(\tau)}
- \pi i
+ 24 \pi i \sum_{r = 1}^{\infty} r \frac{q(\tau)^r}{1 - q(\tau)^r} \label{eq:7}\\
&=&\frac{3 i}{\Im(\tau)}
- \pi i
+ 24 \pi i \sum_{r = 1}^{\infty} r \sum_{s=1 }^{\infty }  q(\tau)^{rs} \nonumber \\
&=&\frac{3 i}{\Im(\tau)}
- \pi i
+ 24 \pi i \sum_{r = 1}^{\infty}  \sigma_1(r) q(\tau)^r
\nonumber \\
&=& -\pi i \east(\tau). \nonumber 
\end{eqnarray}
\end{proof}

\section{First and second minima of Faltings' height}
\label{reduccion extremal}
In this section we prove Theorem~\ref{t:minima} assuming Proposition~\ref{p:minima}.
The proof is in Section~\ref{seccion esencial}, after we give in Section~\ref{valor minimo} a proof of~\eqref{gmin} and of the fact that the minimum value of~$\Fh$ is~$\Fh(0)$.

In what follows we use the following formula of~$\Fh$.
First, for each prime number~$p$ fix an extension~$| \cdot |_p$ to~$\overline{\Q}$ of the~$p$-adic norm on~$\Q$.
Furthermore, consider the action of the Galois group~$\Gal \left( \overline{\Q} / \Q \right)$ on~$\overline{\Q}$ and for~$\alpha$ in~$\overline{\Q}$ denote by~$\cO(\alpha)$ the orbit of~$\alpha$.
Then, choosing $s$ in ~\eqref{eq:9} as the section corresponding to $\Delta \in M_{12}\left(SL_2(\z)\right)$ through \eqref{modforms}, we have by Lemma \ref{magia}
\begin{equation}
  \label{e:global/local}
\Fh(\alpha)
=
\frac{1}{12} \left( \frac{1}{\# \cO(\alpha)} \sum_{\alpha' \in \cO(\alpha)} \ghyp(\alpha')
+
\frac{1}{\# \cO(\alpha)} \sum_{p \text{ prime}} \sum_{\alpha' \in \cO(\alpha)} \log^+ |\alpha'|_p \right).
\end{equation}

Throughout this section we set~$\rho \coloneqq e^{\pi i / 3}$ and denote by
\begin{equation}\label{dominio fundamental}
 T\coloneqq\left\{\tau \in \H : \left| \Re(\tau) \right| \leq \frac{1}{2}, |\tau|\geq 1 \right\}
\end{equation}
the closure of the standard fundamental domain for the action of $\SL_2(\Z)$ on $\H$. 

\subsection{Minimum value of Faltings' height}\label{valor minimo}
In this section we prove~\eqref{gmin} and the fact that the minimum value of~$\Fh$ is~$\Fh(0)$.

The first equality in~\eqref{gmin} is a direct consequence of~\eqref{e:global/local} and~$j(\rho) = 0$ and the second one is a direct consequence of Lemma~\ref{Edos}, \ref{ddelta}.
To show that the minimum value of~$\Fh$ is~$\Fh(0)$, consider the lower bound
$$ \inf \left\{ \Fh(\alpha) : \alpha \in \overline{\Q} \right\}
\ge
\frac{1}{12} \inf \left\{ \ghyp(\zeta) : \zeta \in \C \right\}, $$
which follows trivially from~\eqref{e:global/local}.
Since by~\eqref{e:global/local} we also have~$\Fh(0) = \frac{1}{12} \ghyp(0)$, the following lemma implies that the minimum value of~$\Fh$ is~$\Fh(0)$.
\begin{lemm}
\label{l:really increasing}
For every~$\tau$ in~$T$, we have~$\gH(\tau) \ge \gH \left(\frac{1}{2}  +
  i \Im(\tau) \right)$, with equality if and only if~$\Re(\tau) =\frac{1}{2}$.
Moreover, the function~$t \mapsto \gH\left(\frac{1}{2} + i t\right)$ is strictly increasing
on~$\left[ \frac{\sqrt{3}}{2}, + \infty \right)$.
In particular, the function~$\ghyp$ attains its minimum value at, and only at, $\zeta = 0$.
\end{lemm}
\begin{proof}
To prove the first statement,  fix $\tau \in T$ and define a $1$\nobreakdash-periodic, smooth function $l \colon \R \rightarrow \R $ by $l(s)\coloneqq\gH(s+i\Im(\tau)).$ Since $\gH$ is real valued, using~\eqref{integral} we have that 
$$l'(s)= 2 \Re \left(\partial \gH(s+i\Im(\tau))\right)= 2 \pi \Im \left(\east (s+i\Im(\tau))\right).$$
Hence, by Lemma~\ref{realidad Eisenstein}, we conclude that the maximum and minimum values of $l(\cdot)$ are attained at $s\in \left\{0,\frac{1}{2}\right\}$. Then, the desired inequality $l(0) > l(\frac{1}{2})$ is equivalent to $|\Delta(i\Im(\tau))|\leq |\Delta(\frac{1}{2}+i\Im(\tau))|$, and this is clear from the product formula for $\Delta$. 

To prove the second statement, note that  the function~$h \colon
(0, + \infty) \to \R$ defined by~$h(t) \coloneqq \gH\left(\frac{1}{2} + i t\right)$ satisfies
$$ h'(t)
=
- 2 \Im \left( \partial \gH\left(\frac{1}{2} + i t\right) \right)
= 2\pi \Re \left(\east\left(\frac{1}{2} + i t\right)\right)=2\pi \east\left(\frac{1}{2} + i t\right). $$
The last equality easily follows from the definition of $\east$. In particular, $h'$ is continuous and~$\lim_{t \to + \infty} h'(t) =
2\pi$.
The desired statement follows from the fact that~$h'(t)$ does not
vanish on~$\left( \frac{\sqrt{3}}{2}, + \infty \right)$, because the
function~$\east$ vanishes only at the orbits of~$i$ and~$\rho$, \emph{cf.} Lemma~\ref{realidad Eisenstein}.
\end{proof}

\subsection{Second minimum of Faltings' height}
\label{seccion esencial}
In this section we prove Theorem~\ref{t:minima} assuming Proposition~\ref{p:minima}.
We postpone the proof of Proposition~\ref{p:minima} to Section~\ref{key}.

From the product formula for the modular discriminant we deduce the asymptotic expansion
 
\begin{equation}\label{log log}
 \gH(\tau)=-\log |q| -6\log(-\log|q|)+  O(1), \quad \Im(\tau)\rightarrow \infty, \quad q=e^{2\pi i \tau}.
\end{equation}

\noindent Since $j(\tau)=\frac{1}{q}+O(1)$ when $\Im(\tau)\rightarrow \infty$, we infer from the definition of $\ghyp$ in~\eqref{def ghyp}  the asymptotic expansion

\begin{equation}\label{asintotica}
 \ghyp(z)=\log|z|-6\log(\log|z|) +O(1)  \textrm{ as }|z|\rightarrow \infty.
\end{equation}

On the other hand, the function $\ghyp$ is invariant under complex conjugation. More precisely, 
\begin{equation}\label{invarianza}
 \ghyp(\overline{z})=\ghyp(z), \quad \textrm{ for all } z \in \C.
\end{equation}
Indeed, choose $\tau \in \h$ with $j(\tau)=z$.
Since the coefficients in the $q$-expansion of~$j$ and~$\Delta$ are real, we have the identities
\begin{equation}\label{conju}
\overline{j(\tau)}=j(-\overline{\tau}), \quad \overline{\Delta(\tau)}=\Delta(-\overline{\tau}).
\end{equation}
  Then,

$$\ghyp(\overline{z})=\ghyp \circ j(-\overline{\tau})= \gH(-\overline{\tau}).$$

\noindent Since $\Im(-\overline{\tau})=\Im(\tau)$ and $|\Delta(-\overline{\tau})|=|\overline{\Delta(\tau)}|=|\Delta(\tau)|$, we have that 

$$\gH(-\overline{\tau})=\gH(\tau)=\ghyp \circ j (\tau)=\ghyp(z),$$ justifying~\eqref{invarianza}.

\begin{proof}[Proof of Theorem~\ref{t:minima}, assuming Proposition~\ref{p:minima}]
By~\eqref{e:global/local} and Lemma~\ref{l:really increasing} we have
$$ \Fh(1)
=
\frac{1}{12} \ghyp(1)
>
\frac{1}{12} \ghyp(0)
=
\Fh(0). $$
Thus, to prove the theorem it is enough to show that there is~$\kappa > 0$ such that for every algebraic number~$\alpha \neq 0, 1$ we have~$\Fh(\alpha) \ge \Fh(1) + \kappa$.
To do this, we essentially apply, for a sufficiently small~$\varepsilon > 0$, Proposition~\ref{prop:1} with~$s = (j - 1)^{\varepsilon} j^{\partial_x \ghyp(1)} \Delta$.

By Proposition~\ref{p:minima}, we have~$1 - \partial_x \ghyp(1) > 0$.
For each~$\varepsilon$ in~$(0, 1 - \partial_x \ghyp(1))$, let~$G_{\varepsilon} \colon \C \setminus\{0,1\}\to \R$ be
defined by
$$ G_{\varepsilon}(z)
\coloneqq
g_1(z) - \varepsilon \log|z - 1|
=
\ghyp(z) - \partial_x \ghyp(1) \cdot \log |z| - \varepsilon \log|z - 1|, $$
and for each prime number~$p$, let~$G_{\varepsilon, p} \colon \C_p \setminus\{0,1\} \to
\R$ be defined by
$$ G_{\varepsilon, p}(z)
\coloneqq
\log^+|z|_p - \partial_x \ghyp(1) \cdot \log |z |_p - \varepsilon
\log|z - 1|_p. $$
Since~$\partial_x \ghyp(1) + \varepsilon < 1$ and $\partial_x \ghyp(1) > 0$, the function~$G_{\varepsilon, p}$ is nonnegative.
Then by~\eqref{e:global/local} and by the product formula, for every~$\alpha$ in~$\overline{\Q} \setminus \{0, 1 \}$ we have
$$ 12 \Fh(\alpha)
=
\frac{1}{\# \cO(\alpha)} \sum_{\alpha' \in \cO(\alpha)} G_\eps(\alpha')
+ \frac{1}{\# \cO(\alpha)} \sum_{p \text{ prime}} \sum_{\alpha' \in \cO(\alpha)} G_{\eps,p}(\alpha'). $$
Since for each prime~$p$ the function~$G_{\varepsilon, p}$ is
nonnegative, to prove the theorem it is enough to show that
\begin{equation}
  \label{eq:5}
  \inf_{\C\setminus\{0,1\}} G_{\varepsilon} > \ghyp(1).
\end{equation}
Using the asymptotic of~$\ghyp$ given by~\eqref{asintotica}, it follows that there
are~$\varepsilon_0 > 0$ and~$R_0 > 0$ such that for each~$\varepsilon$ in~$(0, \varepsilon_0)$ and each~$z$ in~$\C$
satisfying~$|z| > R_0$, we have
$$G_{\varepsilon}(z) \ge \ghyp(1) + 1.$$
By Proposition~\ref{p:minima}, there is~$\varepsilon\in (0,\eps_0)$ such that for some~$\delta > 0$ and every~$z$ satisfying~$|z - 1| \ge 1/2$ and~$|z| \le R_0$, we have
$$ G_{\varepsilon}(z) \ge \ghyp(1) + \delta. $$
Finally, using Proposition~\ref{p:minima} again, for each~$z$ in~$\C$ satisfying~$|z - 1| \le 1/2$,
we have
$$ G_{\varepsilon}(z)
=
g_1(z) - \varepsilon \log|z - 1|
\ge
\ghyp(1) + \varepsilon \log 2.$$
This completes the proof of~\eqref{eq:5} and of the theorem.
\end{proof}

\begin{rema}
\label{multiplicador}
For any given $\zeta \in \C$, we write $\zeta = x+iy$ the real and imaginary parts.
For a real number~$a$, set $c(a)\coloneqq\inf_{\zeta \in \C} \ghyp(\zeta)-a\log|\zeta|$ and note that by Proposition~\ref{prop:1} with~$s = j^a \Delta$, 
\begin{equation}\label{ca}
 \Fmuess\geq c(a)/12.
\end{equation}
Using Proposition~\ref{p:minima},  we see that for any choice of~$a$ we have $c(a) \leq \inf_{|\zeta| = 1}\ghyp(\zeta)=\ghyp(1)$. Hence, the bound $\Fmuess\geq \ghyp(1)/12=\Fh(1)$ is the best we can hope for using~\eqref{ca}. In order to identify the value of $a$ such that $c(a)=\ghyp(1)/12$, we impose that $\zeta = 1$ is a critical point of $\ghyp(\cdot)-a\log |\cdot|$, thus finding that necessarily $a= \partial_x \ghyp(1) $. 
\end{rema}

\section{Distortion estimates}\label{Koebe}
In this section we estimate the inverse of~$j$ on a suitable neighborhood of the unit disk.
After recalling the Koebe distortion theorem and some of its variants below, we explain the set up in Section~\ref{ss:set up} and then we proceed to the estimates in Section~\ref{ss:j^-1}.

\begin{Thm}\label{Koebe Thm}
Let $\D$ be the open unit disk, and let $f_0 \colon \D \rightarrow \C$ be an univalent (\emph{i.e.} holomorphic and injective) function  such that $f_0(0)=0$ and $ f_0'(0)=1$.
Then for every $w \in \D$,

\begin{enumerate}
\item\label{unoKoebe}
$$ \frac{|w|}{(1+|w|)^2} \leq |f_0(w)| \leq \frac{|w|}{(1-|w|)^2},
\qquad
\frac{1-|w|}{(1+|w|)^3} \leq |f_0'(w)| \leq \frac{1+|w|}{(1-|w|)^3}. $$
\item\label{dosKoebe}
$$ \left| w \frac{f_0'}{f_0}(w)\right| \leq \frac{1+|w|}{1-|w|},
\qquad
\left| w \frac{f_0''}{f_0'}(w) - \frac{2|w|^2}{1-|w|^2}\right| \leq \frac{4|w|}{1-|w|^2}. $$
\item\label{tresKoebe}
$$ |f_0(w)-w|\leq \frac{|w|^2(2-|w|)}{(1-|w|)^2},
\quad
\left|w\frac{f_0'}{f_0}(w)-1\right| \leq 2 |w| \frac{(1+|w|)^2}{(1-|w|)^3}. $$
\end{enumerate}
\end{Thm}

\begin{proof} Parts~\ref{unoKoebe} and~\ref{dosKoebe} are proved in \cite[Lemma~1.3 and Theorem~1.6]{Pommerenke75}.
Part~\ref{tresKoebe} is undoubtedly well-known but we provide a proof due to lack of suitable reference. Write $f_0(w)=\sum_{n=1 }^{\infty }a_nw^n.$ Then, $a_1=1$ and de Branges' theorem ensures that $|a_n|\leq n$ for all~$n$, see for example~\cite{Pommerenke75}.
Hence,

$$|f_0(w)-w|=\left|\sum_{n = 2}^{\infty} a_nw^n\right|\leq \sum_{n = 2}^{\infty}n |w|^n=\frac{|w|^2(2-|w|)}{(1-|w|)^2}.$$

Similarly,
$$|wf_0'(w)-f_0(w)|=\left|\sum_{n = 2}^{\infty} a_n(n-1)w^n\right|\leq |w|^2\sum_{n = 2}^{\infty} (n-1)n|w|^{n-2} =\frac{2|w|^{2}}{(1-|w|)^3}.  $$  
This estimate, combined with part~\ref{unoKoebe} finishes the proof.
\end{proof}

\subsection{Set up}
\label{ss:set up}

Since the $j$-invariant is injective when restricted to a fundamental domain, we aim to use Theorem~\ref{Koebe Thm} to deduce an approximation of it by a rational function  on a neighborhood of $\rho \coloneqq \frac{1 + i \sqrt{3}}{2}$.

In order to transport the situation to a disk, consider the function~$\psi \colon \D \to \H$ defined by

\begin{equation}\label{def psi}
 \psi(w)
\coloneqq
\frac{\overline{\rho} w + \rho}{w + 1},
\end{equation}

\noindent and let $\whf \colon \D
\to \C$ be the function defined by 
\begin{equation}\label{fgorro}
\whf \coloneqq j \circ \psi.
\end{equation}

Consider the following fundamental domain for the action of~$\SL_2(\Z)$ on~$\h$,
\begin{equation*}
T_0 \coloneqq \left\{ \tau \in \h : 0 \leq \Re (\tau)\leq \frac{1}{2}, \quad |\tau-1| > 1 \right\}
\setminus
\left\{it : 0<t\leq 1\right\}.
\end{equation*}

\begin{figure}[h]\label{mono}

\centering
\begin{tikzpicture}
 \draw [ultra thick] (0,0) -- (0.6,0);
 \node at (2,0)[scale=0.8] {$j\geq 1728$};
  \draw [densely dotted, thick] (0,1) -- (0.6,1);
 \node at (2,1)[scale=0.8] {$0 \leq j\leq 1728$};
  \draw [dashed] (0,2) -- (0.6,2);
 \node at (2,2)[scale=0.8] {$j\leq 0$};
 \draw  (0,3) -- (0.6,3);
 \node at (2,3)[scale=0.8] {$0 \leq j\leq 1728$ };
\fill[gray] (0.53,2.92) rectangle(0.67,3.08); 
\fill[black] (-0.08,2.94) rectangle(0.08,3.06); 
\node at (4,0) {$ $};
\end{tikzpicture}
\begin{tikzpicture}[scale=1.5]
\path(1,1) coordinate(Q);
\path (1,0) coordinate(P);
\path (60:1cm) coordinate(rho);

\draw (0,0) -- (1,0);
\draw (0,0) -- (0,3);
\draw (0.5,0) -- (rho);
\draw [dashed] (rho) -- (0.5,3);
\draw [ultra thick] (0,1) -- (0,3);

\draw (P) arc(0:60:1);
\draw (rho) [densely dotted, thick] arc(60:90:1);
\draw (Q) arc(90:180:1);

\path (60:1.2cm) coordinate(rhorho);
\node at (rhorho)  {$\rho$};

\fill[black] (60:0.93cm) rectangle(60:1.08cm);

\node at (0.2,1.1)  {$i$};
\fill[gray] (0,1) circle(1.3pt);

\node at (0.7,0.5)[scale=0.8]  {$\frac{1+i}{2} $};

\fill[gray] (0.45,0.45) rectangle(0.55,0.55);

\node at (0.25,2.2) {$*$};
\node at (0.25,1.5) {$*$};
\node at (0.2,0.83) {$\circ$};
\node at (0.1,0.62) {$\circ$};

\end{tikzpicture}
\begin{tikzpicture}[->,>=stealth',auto,node distance=1cm,
  thick,scale=1.5]
\node at (0,0) {};
  \node (1) at (0,1) {};
  \node (2) [right of=1] {$\psi$};
  \node (3) [right of=2] {};

  \path[every node/.style={font=\sffamily\small}]
    (3) edge[bend right] node [left] {} (1);
 
\end{tikzpicture}
\begin{tikzpicture}[scale=1.7]

\path (60:1cm) coordinate(R);
\path (-60:1cm) coordinate(Rbar);
\path (240:1cm) coordinate(RR);
\path (240:0.26794919243cm) coordinate(RM);
\path (0.26794919243cm,0) coordinate(rcero);

\draw (0,0) circle(1);
\draw (0,0) circle(0.26794919243cm);
\draw [dashed] (-1,0) -- (0,0);
\draw (0,0) -- (1,0);
\fill[black] (-0.03,-0.06) rectangle(0.03,0.06);

\draw (R) -- (0,0);
\draw [densely dotted, thick] (0,0) -- (RM);
\draw (RM) -- (RR);
\draw (0,0) -- (Rbar);

\coordinate (B) at (-1,0);
\arcThroughThreePoints{Rbar}{RM}{B};

\node[scale=0.8] at (210:0.2cm) {$*$};
\node[scale=0.8] at (-0.4,-0.06) {$*$};
\node[scale=0.8] at (270:0.18cm) {$\circ$};
\node[scale=0.8] at (290:0.35cm)  {$\circ$};

\path (60:1.3cm) coordinate(Runo);
\node at (Runo)  {$\rho$};
\fill[black] (R) circle(1.3pt);

\path (-60:1.3cm) coordinate(Rbaruno);
\node at (Rbaruno)  {$\overline{\rho}$};
\fill[black] (Rbar) circle(1.3pt);

\path (0.5cm,-0.2) coordinate(rcerouno);
\node at (rcerouno)  {$r_0$};
\fill[gray] (0.22,-0.05) rectangle(0.32,0.05);

\path (228:0.6cm) coordinate(RMuno);
\fill[gray] (RM) circle(1.3pt);

\end{tikzpicture}
\caption   {Action of $\psi$ on the fundamental region.}
\end{figure}
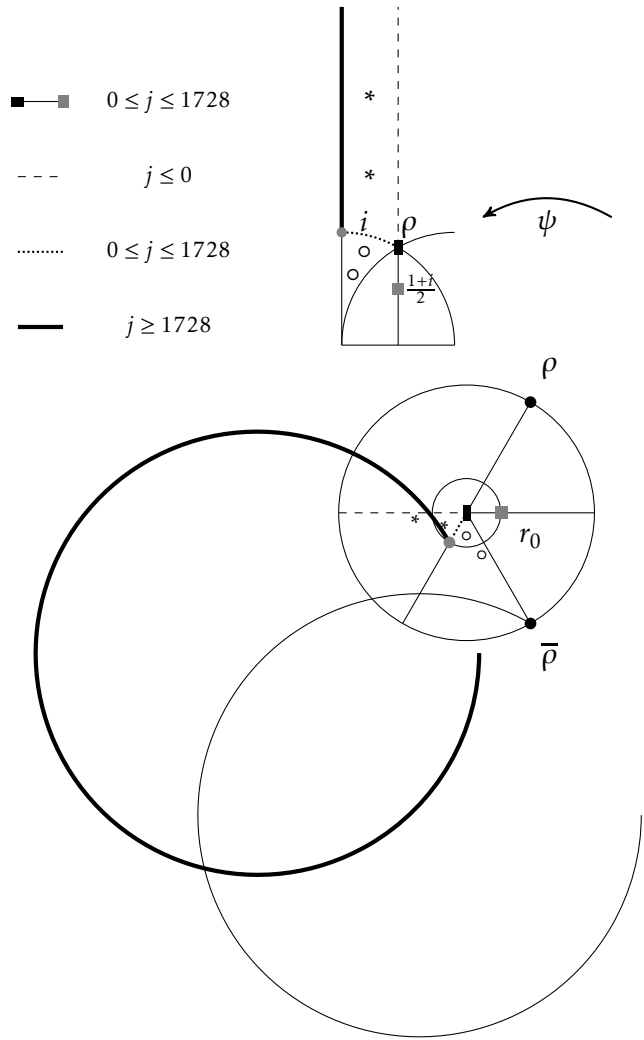

\begin{lemm}\label{dominio}
Define
$$ r_0 \coloneqq 2 - \sqrt{3}, \quad B(0,r_0)\coloneqq\{w \in \C: |w|<r_0\}, $$
and 
$$ B^*\coloneqq\{ z \in B(0,r_0) : \arg z \in [\pi, 5\pi/3)\}. $$ 
Then, we have that $\psi(B^*) \subseteq T_0$.
\end{lemm}

\begin{proof} Let $\ph\colon \h \rightarrow \D$ be the inverse of the function $\psi$, which is given by $$\ph(\tau)\coloneqq-\frac{\tau-\rho}{\tau-\overline{\rho}}.$$ 

We show the equivalent assertion  $B^*\subseteq \ph(T_0)$. Since $\ph$ is a conformal mapping, it is enough to study the image of the boundary in $\h$ of $T_0$, which is the union of the three sets
$$L_0\coloneqq\{it: t >0\}, \quad L_1\coloneqq\left\{ \frac{1}{2} + it : t\geq \frac{\sqrt{3}}{2} \right\}, \quad C\coloneqq\left\{ \tau \in \h : |\tau-1|=1, 0 \leq \Re(\tau) \leq \frac{1}{2} \right\}.$$ 

Since $\ph$ is a M\"obius transformation, the three sets are sent into line or circle segments.
Noting that $\ph(\overline{\rho})=\infty$, we find

$$\ph(L_1)=(-1,0], \quad  \ph(C)=\{t\overline{\rho} : 0\leq t <1\}.$$ Let $R$ be the circle  that passes through the points $\{\overline{\rho}, \ph(i), -1\}$. Then, $\ph(L_0)$ is the open arc of $R$ that contains $\ph(i)$ and has  extreme points  $\overline{\rho}$ and -1.
 
A calculation shows that  $\ph(i)=-r_0\rho$. We conclude the proof by observing that $\arg(\overline{\rho})=\frac{5\pi}{3}.$
\end{proof}

Note that~$\psi(0) = \rho$, and that~$\whf \colon \D \to \C$ is
invariant under the rotation~$z \mapsto - \rho z$.
It follows that there is a holomorphic function~$f \colon \D \to \C$ such that for every~$w$
in~$\D$ we have
\begin{equation}\label{def f}
\whf(w) = f(w^3).
\end{equation}

\begin{lemm}\label{radio}
Let $r_0=2-\sqrt{3}$. The function~$f$ defined in~\eqref{def f} is
univalent on~$B \left( 0, r_0^3 \right)$. In addition, we have that $f'(0)=i\frac{\sqrt{3}}{2}\cdot j'''(\rho)= \left( \frac{\sqrt{3}}{\pi} \Gamma \left(\frac{1}{3} \right)^2 \right)^9$. In particular, $f'(0)$ is a real number and

 $$237698 \leq f'(0)\leq  237699.$$
\end{lemm}

\begin{proof}
We first show that $f$ is injective. Let $w_1,w_2\in B(0,r_0^3)$ be such that $f(w_1)=f(w_2)$. Choose $z_1, z_2 \in B(0,r_0)$ such that 
$$ z_i^3=w_i, \quad \arg z_i \in [\pi, 5\pi/3), \quad  i=1,2.$$

\noindent Then, we have that $\whf(z_1)=\whf(z_2)$, implying $$j\left( \psi (z_1)\right) = j\left( \psi (z_2)\right).$$ 
Since Lemma~\ref{dominio} ensures that $ \psi (z_1),  \psi (z_2) \in T_0$ and the $j$-invariant is injective on any fundamental domain, we conclude $ \psi (z_1)=  \psi (z_2)$, whence $z_1=z_2$, showing that $w_1=w_2.$

Using Lemma~\ref{valores especiales} and  $\psi'(0)=-i\sqrt{3}$, we find that 
\begin{multline}
f'(0)
=
\frac{1}{6} \whf'''(0)
=
\frac{1}{6} j'''(\rho) \psi'(0)^3
=
\frac{1}{6} \left(- 2^{10}3 \pi^3 i E_6(\rho) \right) (-i\sqrt{3})^3
\\ =
2^9 3\sqrt{3} \pi^3 E_6(\rho)
=
\left( \frac{\sqrt{3}}{\pi} \Gamma \left(\frac{1}{3} \right)^2 \right)^9
=
237698.411625786...  
\end{multline}
\end{proof}

Next we apply the distortions statements in Theorem~\ref{Koebe Thm} several times to~$f|_{B \left(0, r_0^3 \right)}$.
To normalize this function, throughout the rest of this section we put~$\varepsilon_1 \coloneqq \left( r_0^3 f'(0) \right)^{-1}$,
and let~$f_0 \colon \D \to \C$ be the function defined by
\begin{equation}
  \label{eq:11}
  f_0(z) \coloneqq \varepsilon_1 f(r_0^3 z).
\end{equation}
It is univalent and satisfies~$f_0(0) = 0$ and~$f_0'(0) = 1$.
\begin{lemm}\label{Mas Koebes}
For every~$z$ in~$\D$, we have the inequalities
 $$\left| \frac{f'}{f}(r_0^3z)r_0^3z \right| \leq \frac{1+|z|}{1-|z|}, \quad \left| \frac{f'}{f}(r_0^3z)r_0^3z -1 \right| \leq 2|z|\frac{(1+|z|)^2}{(1-|z|)^3}$$ 
 and
 $$\left| \frac{f''}{f}(r_0^3z)r_0^6z^2\right| \leq \frac{2|z|(2+|z|)}{(1-|z|)^2}.$$
\end{lemm}

\begin{proof}
Note that the function~$f_0$ defined by~\eqref{def f} satisfies
$$\frac{f_0'}{f_0}(z)=r_0^3\frac{f'}{f}(r_0^3z)
\text{ and }
\frac{f_0''}{f_0'}(z)=r_0^3\frac{f''}{f'}(r_0^3z). $$
Hence, the first and second asserted inequalities are a direct consequence of Theorem~\ref{Koebe Thm}, \ref{dosKoebe} and~\ref{tresKoebe}. On the other hand, we can use Theorem~\ref{Koebe Thm}, \ref{dosKoebe}, again to obtain 

$$\left|r_0^3z\frac{f''}{f'}(r_0^3z)\right| \leq  \left|z\frac{f_0''}{f_0'}(z)-\frac{2|z|^2}{1-|z|^2}\right| + \frac{2|z|^2}{1-|z|^2}\le  \frac{2|z|(2+|z|)}{1-|z|^2}.$$

\noindent Then,
$$ \left| \frac{f''}{f}(r_0^3z)r_0^6z^2\right| = \left| r_0^3z \frac{f''}{f'}(r_0^3z)\cdot  \frac{f'}{f}(r_0^3z)   r_0^3z\right| \leq \frac{2(2+|z|)|z|}{(1-|z|)^2}, $$

\noindent as desired. 
\end{proof}

\subsection{Approximating the inverse of $j$ on a neighborhood of the unit disk}
\label{ss:j^-1}

Here, we provide estimates on $\whf^{-1}$ on a neighborhood of the unit circle.
To this end, for each~$\alpha$ in $\left(0, \frac{1}{4\eps_1} \right)$ denote by $\kappa(\alpha)$ the smallest solution~$x$ of 
$$ 1 + x
=
\alpha\left(1+(1+x)\eps_1\right)^2,$$
which is given explicitly by $$\kappa(\alpha)=\frac{1}{2\eps_1^2}\cdot \left( \frac{1}{\alpha} -2\eps_1(1+\eps_1)-\sqrt{\frac{1}{\alpha}\left(\frac{1}{\alpha}-4\eps_1\right)  } \right),$$
and put
$$r_+(\alpha) \coloneqq \left(\frac{1+\kappa(\alpha)}{f'(0)}\right)^{1/3},
\text{ and }
r_-(\alpha) \coloneqq (1-4\alpha\eps_1)^{1/3}r_+(\alpha). $$
In the rest of this section we denote
$$ \kappa_1\coloneqq \kappa(1), \quad r_1^+ \coloneqq r_+(1), \quad  r_1^- \coloneqq r_-(1). $$
Noting that
\begin{equation}\label{defkappa}
\kappa_1 =
2 \varepsilon_1 \left( \frac{2 + \varepsilon_1}{1 -
    2\varepsilon_1 - 2 \varepsilon_1^2 + \sqrt{1 - 4 \varepsilon_1}}
\right),
\end{equation}
and that by Lemma~\ref{radio}
\begin{equation}
  \label{eq:15}
  \frac{1}{4573}
\le
\varepsilon_1
\le
\frac{1}{4572},
\end{equation}
we have
\begin{equation}
  \label{eq:16}
2\varepsilon_1
\le
\kappa_1
\le
2 \varepsilon_1 (1 + 3 \varepsilon_1)
\le
\frac{1}{2284}.  
\end{equation}

\begin{lemm}\label{l:w uno} 
Let $w \in B(0,r_0)$ and put  $\zeta=\whf(w)$. Then, we have that $r_-(|\zeta|)\leq |w| \leq r_+(|\zeta|)$. In particular, if $|\whf(w)|=1$, then $r_1^- \le |w| \le r_1^+$.
\end{lemm}

\begin{proof}
Applying Theorem~\ref{Koebe Thm}, \ref{unoKoebe},  to~$f_0$, it follows that for each~$z$ in~$\D$ satisfying~$|z| =
\left(1 + \kappa(|\zeta|)\right)\varepsilon_1$, we have by the definition of~$\kappa(\cdot)$
$$ \left| f \left( r_0^3z \right) \right|
=
\frac{|f_0(z)|}{\varepsilon_1}
\ge
\frac{1}{\varepsilon_1} \cdot \frac{|z|}{(1 + |z|)^{2}}
=
\frac{1 + \kappa(|\zeta|)}{\left(1 + \left(1 + \kappa(|\zeta|)\right)\varepsilon_1\right)^2}
=
|\zeta|. $$

Hence, the domain bounded by the Jordan curve $f\left( \partial B\left(0, \frac{1 + \kappa(|\zeta|)}{f'(0)}\right)  \right) $ contains $B(0,|\zeta|)$. Since $\zeta=f(w^3)$, it follows that~$w^3$ is in~$B\left( 0, \frac{1 + \kappa(|\zeta|)}{f'(0)}
\right)$. This proves the second desired inequality.

To prove the first inequality, we apply Theorem~\ref{Koebe Thm}, \ref{unoKoebe}, to~$f_0$
and~$z = \left(\frac{w}{r_0}\right)^3$. The inequality we have just proved implies $|z|\leq \left(1 + \kappa(|\zeta|)\right)\varepsilon_1$. Hence,  we obtain
$$ |\zeta|
=
\left|f \left( w^3 \right)\right|
=
\frac{1}{\varepsilon_1} \left| f_0 \left( \left( \frac{w}{r_0} \right)^3 \right) \right|
\le
\frac{1}{\varepsilon_1} \cdot \frac{|z|}{(1 - |z|)^2}
\le
\frac{f'(0)|w|^3}{\left(1 - \left(1 + \kappa(|\zeta|)\right)\varepsilon_1\right)^{2}}. $$

\noindent Then, by the definition of $\kappa(\cdot)$, $$|w|^3 \geq |\zeta| \frac{\left(1-\left(1+\kappa(|\zeta|)\right)\eps_1\right)^2}{f'(0)}=(1-4|\zeta|\eps_1)\cdot \frac{1+\kappa(|\zeta|)}{f'(0)}  .$$
This proves the first inequality, and completes the proof of the lemma.
\end{proof}

\begin{lemm}\label{zetaw}
Let $\zeta \in S^1$ and let $w \in B(0,r_0^3)$ be such that $\whf(w)=\zeta$. Then, we have that

\begin{enumerate}
 \item
\label{le:linear}
$|\zeta-f'(0)w^3|\leq \frac{\eps_1(1+\kappa_1)^2\left(2-\eps_1(1+\kappa_1)\right)}{\left(1-\eps_1(1+\kappa_1)\right)^2}\leq \frac{1}{2283}.$
 \item
\label{le:log}
$\left| \log(1 - |w|^2) - \log \left(1 - f'(0)^{-\frac{2}{3} } \right) \right|
\leq
\frac{2}{3}  f'(0)^{-\frac{2}{3} }\kappa_1\left( 1-(r_1^+)^2\right)^{-1}
\leq
7.7 \cdot 10^{-8}.$
\item
\label{le:j'}
$185 \leq |\whf'(w)| \leq 186.054.$
\end{enumerate}
 
\end{lemm}

\begin{proof}
 Set $z=\frac{w^3}{r_0^3}$.   From the definitions and using part~\ref{tresKoebe} of Theorem~\ref{Koebe Thm} we have that 
 
 $$ |\zeta-f'(0)w^3|=\frac{1}{\eps_1}|f_0(z)-z|\leq  \frac{|z|^2(2-|z|)}{\eps_1(1-|z|)^2}.  $$
By Lemma~\ref{l:w uno}, we have that $|z|\leq (1+\kappa_1)\eps_1$.
This estimate, \eqref{eq:16} and Lemma~\ref{radio} justify the first assertion.

In view of Lemma~\ref{l:w uno}, we have that 
\begin{align*}
 \left|\log(1-|w|^2)-\log(1-f'(0)^{-\frac{2}{3} })\right|  & \leq (|w|^2-f'(0)^{-\frac{2}{3} })\cdot \frac{1}{1-(r_1^+)^2} \\
& \leq f'(0)^{-\frac{2}{3} } \left( (1+\kappa_1)^{\frac{2}{3} }-1)\right) \cdot  \frac{1}{1-(r_1^+)^2} \\
& \leq \frac{2}{3}  f'(0)^{-\frac{2}{3} }\frac{\kappa_1}{ 1- (r_1^+)^2} .
\end{align*}

\noindent The second assertion is obtained by evaluating this last quantity.

Using the definition of~$\whf$, we have
$$ \whf'(w)=3w^2f'(w^3)=3w_1^2f'(0)f_0'(z). $$
Then, Theorem~\ref{Koebe Thm}, \ref{unoKoebe}, implies 

$$ \frac{1-|z| }{(1+|z|)^3}    \leq \left|\frac{\whf'(w)}{3w^2f'(0)} \right|\leq \frac{1+|z|}{(1-|z|)^3}. $$

\noindent Lemma~\ref{l:w uno}, ensures that $r_1^-\leq |w|\leq r_1^+.$ Using Lemma~\ref{radio} and~\eqref{defkappa}, we find 

$$|\whf'(w)|
\leq
3(r_1^+)^2 f'(0)\cdot \frac{1+\frac{(r_1^+)^3}{r_0^3} }{\left(1-\frac{(r_1^+)^3}{r_0^3}\right)^3}
=
\frac{3f'(0)^{1/3}(1+\kappa_1)^{2/3}}{(1-4\eps_1)(1-(1+\kappa_1)^2\eps_1^2)} \leq 186.$$ 
 A similar reasoning leads to the lower bound $|\whf'(w)|\geq 185$. 
\end{proof}

\section{Approximating $\gH$ on a neighborhood of the locus~$|j| \le 1$}
\label{s:approximation}
The aim of this section is to provide an approximation of~$\gH$ and of its first and second derivatives, on a suitable neighborhood of the locus~$|j| \le 1$.
This is stated as Proposition~\ref{p:approximation} below.
It is convenient to express this approximation in terms of the function $g_{\D} \colon \D \rightarrow \R$  given by 
\begin{equation}
\label{g en el disco}
g_{\D} \coloneqq \gH \circ \psi,
\end{equation}
where $\psi$ is defined in~\eqref{def psi}.
The approximation is also stated in terms of the derivative~$f'(0) = i \frac{\sqrt{3}}{2} \cdot j'''(\rho)$, computed in Lemma~\ref{radio}, and of the holomorphic function~$\whh \colon \D \to \C$ defined by
$$ \whh(w) \coloneqq \frac{\Delta \circ \psi(w)}{(1 + w)^{12}}. $$
Note that for every~$w$ in~$\D$ we have
$$ \Im(\psi(w))
=
\frac{\sqrt{3}}{2} \cdot \frac{1 - |w|^2}{|1 + w|^2}
$$
and
\begin{equation}\label{sobre gD}
g_{\D}(w)
=
- \log(1728 \pi^6) - 6 \log(1 - |w|^2) - \log|\whh(w)|.
\end{equation}

\begin{prop}
\label{p:approximation}
For every~$w$ in~$\D$ satisfying~$|w| \le 1 - \frac{\pi}{2 \sqrt{3}}$, we have
\begin{eqnarray}
\left| g_{\D}(w) - \left( g_{\D}(0) - 6 \log (1 - |w|^2) -
\frac{f'(0)}{13824} \Re(w^3) \right) \right|
&\le&
6^3 |w|^6, \label{l:approx} \\
\left| (\log \whh)' (w) w - \frac{3\cdot f'(0)}{13824} w^3\right| & \leq & 6^4|w|^6,  \label{quinto orden} \\
\left|  (\log \whh)''(w)w^2-\frac{6\cdot f'(0)}{13824} w^3 \right|& \leq &5 \cdot 6^4|w|^6. \label{cuarto orden}
\end{eqnarray}
\end{prop}

The proof of this proposition boils down to an estimate of sixth order derivative of the holomorphic function~$\log \whh$ (Lemma~\ref{l:sixth order}).
This is done in Section~\ref{ss:approximation}, after establishing some properties of the function~$\whh$ in Section~\ref{ss:decomposition}.

\subsection{Some properties of $\whh$}
\label{ss:decomposition}
Recall that a holomorphic function is \emph{real}, if it is defined on a connected domain that is invariant under complex conjugation and if the function commutes with complex conjugation.
\begin{lemm}\label{conjugaciones}
The functions~$\whf$ and~$\whh$ are both real.
\end{lemm}

\begin{proof} A routine calculation shows that  $-\overline{\psi(w)}=\psi(\overline{w})-1$.  Using~\eqref{conju}, we obtain $$\overline{\whf(w)}=\overline{j \circ \psi (w)}=j\left(-\overline{\psi(w)}\right)=j(\psi(\overline{w})-1)=j\circ \psi (\overline{w}) = \whf(\overline{w}).$$

By the same argument, we have that 
$$\overline{\whh(w)} = \frac{\Delta \left(- \overline{\psi(w)}\right)}{(1 + \overline{w})^{12}}=\frac{\Delta \left( \psi(\overline{w})-1\right)}{(1 + \overline{w})^{12}}=\whh(\overline{w}).$$
\end{proof}

\begin{lemm}
\label{l:holomorphic part}
  There is a holomorphic function~$h \colon \D \to \C$, such that for
  each~$w$ in~$\D$ we have~$h ( w^3) = \whh(w)$.
\end{lemm}

\begin{proof} it is enough to show that $h$ is invariant under the rotation $w \mapsto -\rho w$. Fix~$w$ in~$\D$ and put~$\tau \coloneqq \psi(w)$.
Noting that~$\psi(- \rho w) = \frac{\tau -
  1}{\tau}$, and using that~$\Delta$ is a modular form of weight~$12$, we have
$$ \whh(-\rho w)
=
\frac{\Delta \circ \psi(- \rho w)}{(1 - \rho w)^{12}}
=
\frac{\Delta \left( \frac{\tau - 1}{\tau} \right)}{(1 - \rho w)^{12}}
=
\Delta (\tau) \left(\frac{\tau}{1 - \rho w} \right)^{12}
=
\Delta \circ \psi(w) \left( \frac{\rho}{1 + w} \right)^{12}
=
\whh(w).  $$
\end{proof}

\begin{lemm}\label{tercera}
 We have
 
\begin{enumerate}
 \item \label{dercero} $(-\log \whh)'(0)=(-\log \whh)''(0)=(-\log \whh)^{(4)}(0)= (-\log \whh)^{(5)}(0)=0$
 \item $\frac{ (-\log \whh)'''(0)}{3!}=-i\frac{\sqrt{3}}{2}\cdot \frac{j'''(\rho)}{13824}=    -\frac{f'(0)}{13824}. $
\end{enumerate}

\end{lemm}

\begin{proof} The function $\whh$ does not vanish on $\D$, hence we can choose a branch of the logarithm such that $K_0(w)\coloneqq -\log \whh(w)$ is holomorphic. Lemma~\ref{l:holomorphic part}, ensures that there is a holomorphic function  $K_1(w)$ such that $K_0(w)=K_1(w^3)$. This justifies the first part of the assertion. On the other hand, we have that 

$$K_0'(w) = \frac{12}{1+w} - \left(\frac{\Delta'}{\Delta}\circ \psi(w)\right)\cdot \psi'(w).$$
Note that by~\eqref{integral} and the equation $\partial \gH(\tau)=\frac{3i}{\Im(\tau)} -\frac{1}{2}\cdot \frac{\Delta'(\tau)}{\Delta(\tau)}$,
$$   \frac{\Delta'}{\Delta} = 2\pi i E_2\quad \text{ and }\quad \psi'(w) = -\frac{i\sqrt{3}}{(1+w)^2},$$
We conclude that

$$(1+w)^2K_0'(w) =2 \sqrt{3}\pi \left(\frac{2 \sqrt{3}}{\pi}(1+w) -E_2\circ \psi(w) \right).$$

\noindent Taking holomorphic derivative, we get

$$2(1+w)K_0'(w) +(1+w)^2K_0''(w)=2 \sqrt{3}\pi \left(\frac{2 \sqrt{3}}{\pi} -E_2'\circ \psi(w) \cdot \psi'(w)\right),$$
implying

$$2(1+w)^3K_0'(w) +(1+w)^4K_0''(w)=6\pi \left(\frac{2 }{\pi}(1+w)^2 +iE_2'\circ \psi(w) \right).$$

\noindent Taking holomorphic derivative once more, we get

\begin{equation*}
\begin{split}
6(1+w)^2K_0'(w)+2(1+w)^3K_0''(w)+4(1+w)^3K_0''(w)+(1+w)^4K_0'''(w) \\
=6 \pi \left(\frac{4}{\pi}(1+w) + \frac{\sqrt{3}}{(1+w)^2}E_2''\circ \psi(w) \right).
\end{split}
\end{equation*}

\noindent Setting $w=0$, we obtain

$$K_0'''(0)=24 + 6\sqrt{3}\pi E_2''(\rho).$$

\noindent Then, using Lemma~\ref{Edos}, \ref{primero} and~\ref{eseis}, and Lemma~\ref{radio}, we conclude the proof.
\end{proof}

\subsection{Approximating $g_\D$}
\label{ss:approximation}
In this section we give the proof of Proposition~\ref{p:approximation}.
The following is the main ingredient.

\begin{lemm}
\label{l:sixth order}
For every~$w$ in~$\D$ satisfying~$|w| \le 1 - \frac{\pi}{2 \sqrt{3}}$,
we have
$$ \frac{\left| (\log \whh)^{(6)}(w) \right|}{6!}
\le
6^3. $$
\end{lemm}

The proof of this lemma is given after the following one.

\begin{lemm}
\label{l:divisors estimate}
For every integer~$n \ge 4$, we have
$$ n^5 \sigma_1(n)
\le
\left( 3^5 \cdot 4 \right) 14^{n - 3}. $$
\end{lemm}

\begin{proof} Since~$\sigma_1(n) \le \frac{n(n + 1)}{2}$, it is enough to prove that
for every~$n \ge 4$ we have
$$ n^6(n +1)
\le
\left( 3^5 \cdot 8 \right) 14^{n - 3}. $$
We proceed by induction.
The case~$n = 4$ can be readily verified.
For the induction step, just note that for each~$n \ge 4$ we have
$$ \frac{(n + 1)^5 (n + 2)}{n^5 (n + 1)}
\le
\frac{5^5 \cdot 6}{4^5 \cdot 5} \le 14.$$
\end{proof}

\begin{proof}[Proof of Lemma~\ref{l:sixth order}.]
Let~$\ell \colon \H \to \C$ be the function defined by
$$ \ell(\tau)
\coloneqq
\left(\frac{\Delta'}{\Delta}\right)(\tau)
=
2\pi i - 48 \pi i \sum_{n = 1}^{\infty} \sigma_1(n) q^n, \quad q=q(\tau)=e^{2\pi i\tau}. $$
Note that for every integer~$k \ge 1$ we have
$$ \ell^{(k)}
=
- 48 \pi i \sum_{n = 1}^{\infty} (2\pi i n)^k \sigma_1(n)q^n, $$
and
$$ \psi^{(k)}(w)
=
(- i \sqrt{3}) k! \left( - \frac{1}{1 + w} \right)^{k + 1}. $$
Defining the complex polynomial
$$ P(\zeta)
\coloneqq
1 - 30 \zeta + 300 \zeta^2 - 1200 \zeta^3 + 1800 \zeta^4 - 720
\zeta^5, $$
and using the formula
\begin{displaymath}
  \begin{split}
(\ell \circ \psi \cdot \psi')^{(5)}
& =
\ell^{(5)} \circ \psi \cdot (\psi')^6
+ 15 \ell^{(4)} \circ \psi \cdot (\psi')^4 \psi''
\\ & \quad
+ 45 \ell''' \circ \psi \cdot (\psi')^2(\psi'')^2
+ 20 \ell''' \circ \psi \cdot (\psi')^3 \psi'''
\\ & \quad
+ 60 \ell'' \circ \psi \cdot \psi' \psi'' \psi'''
+ 15 \ell'' \circ \psi \cdot (\psi'')^3
+ 15 \ell'' \circ \psi \cdot (\psi')^2 \psi^{(4)}
\\ & \quad
+ 10 \ell' \circ \psi \cdot (\psi''')^2
+ 15 \ell' \circ \psi \cdot \psi'' \psi^{(4)}
+ 6 \ell' \circ \psi \cdot \psi' \psi^{(5)}
\\ & \quad
+ \ell \circ \psi \cdot \psi^{(6)},
  \end{split}
\end{displaymath}
we have
\begin{equation}
  \label{eq:13}
  \begin{split}
\frac{(\log \whh)^{(6)}(w)}{6!}
&
= \frac{1}{6!} \left(\frac{\whh'}{\whh} \right)^{(5)} (w)
\\ & =
\frac{2}{(1 + w)^6}
+ \frac{(\ell \circ \psi \cdot \psi')^{(5)}(w)}{6!}
\\ & =
\frac{2}{(1 + w)^6} - \frac{2\sqrt{3} \pi}{(1 + w)^7}
\\ & \quad
- \frac{2^5 3^2 \pi^6}{5 (1 + w)^{12}} \sum_{n = 1}^{\infty} P
\left( \frac{1 + w}{2 \sqrt{3} \pi n} \right) n^5 \sigma_1(n) (q \circ \psi(w))^n.
  \end{split}
\end{equation}

Fix~$w_0$ in~$\D$ satisfying~$|w_0| \le 1 - \frac{\pi}{2 \sqrt{3}}$.
Since~$g_{\D}$ is real and invariant under the rotation~$w \mapsto -
\rho w$, it is enough to consider the case where~$\arg(w_0)$ is
in~$\left[\frac{\pi}{2}, \frac{\pi}{3} \right]$.
This last condition implies that~$\psi(w_0)$ is in~$T$, so, if we
put~$q_0 \coloneqq \exp \left(-\sqrt{3}\pi \right)$,
then~$\left| q \circ \psi(w_0) \right| \le q_0$.
On the other hand, noting that
$$ P(\zeta)
=
(1 - 10 \zeta)^3 - 200(1 - 10 \zeta)(1 + \zeta)\zeta^3 - 2720 \zeta^5, $$
$$ \left| \frac{1 + w_0}{2 \sqrt{3} \pi} \right|
\le
\frac{1}{9}, \quad
\left| 1 - 10 \left( \frac{1 + w_0}{2 \sqrt{3} \pi} \right) \right|
\le
\frac{1}{6}, \quad
\left| 1 - 10 \left( \frac{1 + w_0}{4 \sqrt{3} \pi} \right) \right|
\le
\frac{7}{12}, $$
and that for every integer~$n \ge 3$ we have
\begin{displaymath}
\left| 1 - 10 \left( \frac{1 + w_0}{2 \sqrt{3} \pi n} \right) \right|
\le
1,
\end{displaymath}
we obtain
\begin{equation}
  \label{eq:8}
\left| P \left( \frac{1 + w_0}{2 \sqrt{3} \pi} \right) \right|
\le
\frac{2}{19}, \quad
\left| P \left( \frac{1 + w_0}{4 \sqrt{3} \pi} \right) \right|
\le
\frac{2}{9},
\end{equation}
and for every integer~$n \ge 3$
\begin{displaymath}
\left| P \left( \frac{1 + w_0}{2 \sqrt{3} \pi n} \right) \right|
\le
\frac{91}{90}.  
\end{displaymath}
Together with Lemma~\ref{l:divisors estimate},
\begin{equation}
  \label{eq:1}
|q \circ \psi(w_0)| \le q_0
\text{ and }
  q_0^{-1}
=
\exp \left( \sqrt{3} \pi \right)
>
6^3,
\end{equation}
the last inequality implies
\begin{multline*}
\left| \sum_{n = 3}^{\infty} P \left( \frac{1 + w}{2 \sqrt{3} \pi n}
  \right) n^5 \sigma_1(n) (q \circ \psi(w_0))^n \right|
\le
\frac{91}{90} \sum_{n = 3}^{\infty} n^5 \sigma_1(n) q_0^n
\\ \le
\frac{91}{90} \cdot 3^5 \cdot 4 q_0^3 \frac{14}{13}
=
\frac{2^2 3^3 7^2}{5} q_0^3
\le
\frac{7^2}{2^73^6 \cdot 5}
\le
\frac{5}{2^63^6}.
\end{multline*}
Combined with~\eqref{eq:13}, \eqref{eq:8}, \eqref{eq:1}, and the
inequality~$\frac{\pi}{2} > 15^{\frac{1}{6}}$, this implies
\begin{equation*}
  \begin{split}
    \left| \frac{(\log \whh)^{(6)}(w_0)}{6!} \right|
& \le
\frac{2(2 \sqrt{3})^6}{\pi^6} + \frac{2 \sqrt{3} \pi (2 \sqrt{3})^7}{\pi^7}
\\ & \quad
+ \frac{2^{5} 3^2 \pi^6 (2 \sqrt{3})^{12}}{5 \pi^{12}} \left(
  \frac{2}{19} q_0 + \frac{2}{9} \cdot 2^5 \cdot 3 q_0^2 +
  \frac{5}{2^6 3^6} \right)
\\ & \le
\frac{1728}{\pi^6} \left(
14 + \frac{2^{11} 3^5}{5}
\left( \frac{1}{2^23^3 \cdot 19} + \frac{1}{3^7}
+ \frac{5}{2^6 3^6} \right) \right).
\\ & =
6^3 \cdot \frac{8}{\pi^6} \left(
14 + \frac{2^9 3^2}{5 \cdot 19} + \frac{2^{11}}{3^2 \cdot 5} + \frac{2^5}{3} \right).
\\ & \le
6^3 \frac{8}{\pi^6} \left(
14 + 49 + 46 + 11 \right)
\\ & =
6^3 \left(\frac{2^6 \cdot 15}{\pi^6}\right)
\\ & \le
6^3.   
  \end{split}
\end{equation*}
\end{proof}

\begin{proof}[Proof of Proposition~\ref{p:approximation}]  Inequality~\eqref{l:approx} is equivalent to the assertion

\begin{equation}\label{estimacion}
\left|-\log |\whh(w)|+\log |\Delta(\rho)| +\frac{f'(0)}{13824}\Re(w^3)\right| \leq 6^3|w|^6,  \textrm{ for all } |w| \le 1 - \frac{\pi}{2 \sqrt{3}}.
\end{equation}

Fix $w$ with $0< |w| \le 1 - \frac{\pi}{2 \sqrt{3}}$ and $\zeta \in \ce$ satisfying $|\zeta|=1$. The function $\whh$ does not vanish on $\D$, hence we can choose a branch of the logarithm such that $K_0(w)\coloneqq -\log \whh(w)$ is holomorphic. Let $\alpha\colon(-\varepsilon,1+\varepsilon) \to \R$ be given by $\alpha(t)=\Re \left(\zeta K_0(tw)\right)$. The function $\alpha$ is well defined and smooth if $\varepsilon>0$ is small enough.  We have that $\alpha^{(n)}(t)=\Re \left(\zeta K_0^{(n)}(tw)w^n\right)$ for all $n\geq 0$. 

Using Lemma~\ref{tercera}, \ref{dercero}, and the sixth order Taylor expansion of $\alpha$  we have that

\begin{equation}\label{TVM}
\Re\left(\zeta K_0(w)\right)=\alpha(1)=\alpha(0)  + \frac{1}{3!}\alpha'''(0)  + \frac{1}{6!}\alpha^{(6)}(t^*), 
\end{equation}

\noindent for some  $ 0\leq t^*\leq 1$. Then by Lemma~\ref{l:sixth order}

$$\frac{1}{6!}|\alpha^{(6)}(t^*)|\leq \frac{1}{6!}|K_0^{(6)}(t^*w)|\cdot |w|^6 = \left| \frac{(\log \whh)^{(6)}(t^*w)}{6!} \right|  \cdot |w|^6 \leq 6^3|w|^6.$$

By Lemma~\ref{Edos},\ref{ddelta}, the quantity $\Delta(\rho)$ is a nonzero real number. Then, we have that $$K_0(0)=-\log |\Delta(\rho)|, \quad  \alpha(0)=\Re(-\zeta\log |\Delta(\rho)|).$$ 

\noindent On the other hand $\alpha'''(0)=\Re\left(\zeta K_0'''(0)w^3\right)$, so Lemma~\ref{tercera}   implies~\eqref{estimacion}.

Since $\Re \left[ \zeta(\log \whh)''(w)w^2\right] = -\alpha''(1)$, the same argument, applied all possible $\zeta$ and to the fourth order Taylor expansion of $-\alpha''(\cdot)$, allows us to prove~\eqref{cuarto orden}. Similarly, the identity $\Re \left[\zeta(\log \whh)' (w) w\right]  = -\alpha'(1)$ enables us to use the fifth order Taylor expansion of $-\alpha'(\cdot)$ in order to prove~\eqref{quinto orden}.
\end{proof}

\section{Numerical estimates}\label{numerical}
In this section we prove Theorem~\ref{t:roots of unity}, Corollary~\ref{c:numerical}, and Proposition~\ref{zetacontraw}.
The proof of Proposition~\ref{zetacontraw} is given in Section~\ref{ss:roots of unity}, where we also estimate the values of~$\Fh$ taken at the roots of unity (Corollary~\ref{estimaciones especiales}).
The proofs of Theorem~\ref{t:roots of unity} and Corollary~\ref{c:numerical} are given in Section~\ref{ss:upper bound}.
The main ingredient, besides those developed in the previous sections, is a general way to find upper bounds for the essential minimum (Proposition~\ref{p:upper bound}).
This leads us to make a numerical estimate of the integral~$\ghyp$ over a certain translate of the unit circle.

In the rest of this section we denote by~$\mu$ (resp.~$\phi$) the classical M{\"o}bius (resp. Euler's totient) function.

\subsection{Approximating $\ghyp$ on the unit circle}
\label{ss:roots of unity}
In this section we combine the distortion estimates in Section~\ref{Koebe} with the estimates from Section~\ref{s:approximation} to prove Proposition~\ref{zetacontraw}.
As a consequence, we obtain approximations of values of~$\Fh$ at roots of unity (Corollaries~\ref{estimaciones especiales}). 

\begin{proof}[Proof of Proposition~\ref{zetacontraw}]
Note that by~\eqref{gmin} and Lemma~\ref{radio} we have
\begin{equation}
  \label{eq:10}
  \gamma_1
=
g_{\D}(0) - 6 \log \left(1 - f'(0)^{-\frac{2}{3}} \right).
\end{equation}
So, if $w \in B(0,r_0)$ is such that $\whf(w)=\zeta$, then by Lemma~\ref{l:w uno}, Lemma~\ref{zetaw}, \ref{le:linear} and~\ref{le:log}, and~\eqref{l:approx}, we have
\begin{multline*}
\left| \ghyp(\zeta) - \left( \gamma_1 - \frac{\Re(\zeta)}{13824} \right) \right|
\\
\begin{aligned}
& =
\left|g_\D(w) -\left( g_{\D}(0)-6\log\left(1-f'(0)^{-\frac{2}{3}}\right)- \frac{\Re(\zeta)}{13824}\right) \right|
\\ & \leq
6^3(r_1^+)^6 +6\left|\log(1-|w|^2)-\log\left(1-f'(0)^{-\frac{2}{3} }\right)\right| +\frac{1}{13824}\left|\Re(\zeta)-\Re(w^3)f'(0)\right|
\\ & \leq
6^3\cdot \left(\frac{1 + \kappa_1}{f'(0)} \right)^2  + 6\cdot 7.7 \cdot 10^{-8} +\frac{1}{13824\cdot 2283}
\\ & \leq
5\cdot 10^{-7}.
\end{aligned}
\end{multline*}
\end{proof}

\begin{coro}
\label{estimaciones especiales}
For every integer~$n \ge 1$ and every primitive root of unity~$\zeta_n$ of order~$n$, we have
$$ -0.7486222078 - \frac{1}{165888} \cdot \frac{\mu(n)}{\phi(n)}
\le
\Fh(\zeta_n)
\le
-0.7486221244 - \frac{1}{165888}  \cdot \frac{\mu(n)}{\phi(n)}.$$
In particular,  $-0.748628236 \le \Fh(1) \le -0.748628152$.
\end{coro}
\begin{proof}
Since~$\zeta_n$ is an algebraic integer, by~\eqref{e:global/local} we have
$$ \Fh(\zeta_n)
=
\frac{1}{12 \cdot \phi(n)} \sum_{\zeta \in \cO(\zeta_n)} \ghyp(\zeta). $$ 
Thus, in view of the identity
\begin{displaymath}
\sum_{\zeta \in \cO(\zeta_n) } \zeta = \mu (n),
\end{displaymath}
the corollary is a direct consequence of Proposition~\ref{zetacontraw}.
\end{proof}

\subsection{Estimating the essential minimum from above}
\label{ss:upper bound}
We use the following criterion to estimate the essential minimum~$\Fmuess$ from above.
It is stated for the height~$\Fh$ and the section~$s = \Delta$, but it is clearly valid for general heights and sections.
The proof is based on the classical Fekete-Szeg{\"o} theorem and an equidistribution result shown in~\cite{BPRS}.

\begin{prop}
\label{p:upper bound}
Let~$K$ be a compact subset of~$\C$ that is invariant under complex conjugation and whose logarithmic capacity is equal to~$1$, and denote by~$\rho_K$ its equilibrium measure.
Then there is a sequence of pairwise distinct algebraic integers~$(p_l)_{l \ge 1}$ such that
$$ \lim_{l \to \infty} \Fh(p_l)
=
\frac{1}{12} \int \ghyp \dd \rho_K. $$
In particular, $\Fmuess \le \frac{1}{12} \int \ghyp \dd \rho_K$.
\end{prop}
\begin{proof}
Denote by~$\fM_{\Q} = \{ \text{prime numbers} \} \cup \{ \infty \}$ the set of places of~$\Q$ and for each prime number~$p$ denote by~$\C_p$ the completion of~$(\overline{\Q}, | \cdot |_p)$.
Furthermore, for a point~$\zeta$ in~$\C$, denote by~$\delta_\zeta$ the Dirac mass at~$\zeta$.

By the Fekete-Szeg{\"o} theorem there is a sequence of pairwise distinct algebraic integers~$(p_l)_{l \ge 1}$ such that for each~$l \ge 1$ the set~$\cO(p_l)$ is contained in
$$ \left\{ \zeta \in \C : \text{ there is $\zeta_0 \in K$ such that } |\zeta - \zeta_0| \le 1/l \right\}, $$
see~\cite{FekSze55}.
Note that by~\eqref{e:global/local}, for each~$l \ge 1$ we have
$$ \Fh(p_l)
=
\frac{1}{12 \# \cO(p_l)} \sum_{q \in \cO(p_l)} \ghyp(q). $$
On the other hand, applying \cite[Proposition~7.4]{BPRS} to the closed bounded adelic set formed by~$E_\infty = K$ and for every prime number~$p$ by the unit ball in~$\C_p$, we have that the measure
$$ \frac{1}{\# \cO(p_l)} \sum_{q \in \cO(p_l)} \delta_q $$
converges to~$\rho_K$ in the weak* topology as~$l \to \infty$.
Since the function~$\ghyp$ is continuous, this implies the proposition.
\end{proof}

To obtain a numerical upper bound of~$\Fmuess$, we apply the previous criterion with~$K$ equal to a translate of the unit circle by a real number~$a$.
Our numerical experiments, described in Section~\ref{sec:numer-exper} and in~\cite{BMRan}, suggest that the best choice for the center is~$a = 0.205$, which is what we use in the proof of Corollary~\ref{c:numerical}.
First we give a formula for the corresponding integral.

\begin{lemm}\label{formula integral}
 For a given $t \in [0,1]$ and $a \in (0,2)$, let $w_a(t) \in B(0,r_0)$ be the only complex number with argument in $\left[\pi, \frac{5}{3}\pi \right)$ such that $\whf\left(w_a(t)\right)=a+e^{2\pi i t}$ (\emph{cf}. Figure~\ref{mono}). Similarly, let $s_a \in (0,r_0)$ be the only element  such that $\whf(s_a)=a$. Then,
\begin{equation}\label{formula integrante}
\int_0^1\ghyp \left( a+e^{2\pi it} \right) \dd t
=
- \log (1728 \pi^6) - \log |\whh(s_a)| - 6\int_0^1 \log \left(1 - |w_a(t)|^2 \right) \dd t.
\end{equation}
\end{lemm}
\begin{proof}
Recall the identities $\whf(w)=f(w^3)$ \eqref{def f} and $\whh(w)=h(w^3)$ (Lemma~\ref{l:holomorphic part}). Since~$f$ is univalent (Lemma~\ref{radio}), the inverse function $f^{-1}$ is holomorphic and well defined on the image of $f$. Also, we have the relations
\begin{equation}\label{relaciones}
w_a(t)^3=f^{-1}(a+e^{2\pi it}), \quad  s_a^3=f^{-1}(a).
\end{equation}
In particular, we see that $w_a(\cdot)$ is a continuous function. Since $|w_a(t)|\leq r_0 <1$ for all $t \in [0,1]$, we deduce that the integral in the right-hand side of~\eqref{formula integrante} is well defined.
By~\eqref{sobre gD} and~\eqref{relaciones}, the left-hand side of~\eqref{formula integrante} is equal to
$$ - \log(1728 \pi^6)
- 6 \int_0^1 \log \left( 1 - \left| f^{-1} \left( a + e^{2\pi it} \right) \right|^{2/3} \right) \dd t
- \int_0^1\log \left| h \circ f^{-1} \left( a + e^{2\pi it} \right) \right| \dd t.$$
By Cauchy's formula, 
\begin{eqnarray}
\int_0^1\log \left| h \circ f^{-1} \left(a + e^{2\pi it} \right) \right| \dd t
& = &
\Re \left( \frac{1}{2\pi i}  \int_{\partial B(a,1)}\frac{\log h\circ f^{-1}(w)}{w-a} \dd w \right)
\nonumber \\ & =&
\log |h\circ f^{-1}(a)|. \nonumber 
\end{eqnarray}
Using~\eqref{relaciones} we conclude the proof.
\end{proof}

The following lemma is used to prove the last assertion of Corollary~\ref{c:numerical}.
\begin{lemm}
\label{l:almost integer}
Let~$\alpha$ be a nonzero algebraic number, denote by~$d$ its degree, by~$a$ the leading coefficient of the minimal polynomial in $\z[x]$ of~$\alpha$  and by ~$b$ the constant coefficient. Then, we have that
\begin{equation}
  \label{eq:2}
  \frac{1}{d} \log |a|
\le
\frac{12 (\Fh(\alpha) - \Fh(1))}{1 - \partial_x g_{\hyp}(1)}, \quad  \frac{1}{d} \log |b|
\le
\frac{12 (\Fh(\alpha) - \Fh(1))}{\partial_x g_{\hyp}(1)}.
\end{equation}
\end{lemm}
\begin{proof}
Put~$\omega \coloneqq 1 - \partial_x g_{\hyp}(1)$.
By~\eqref{e:global/local}, the product formula, and the fact that~$b$ is a nonzero integer, we have
\begin{equation*}
  \begin{split}
12 \Fh(\alpha)
-
\frac{1}{d} \sum_{\alpha' \in \cO(\alpha)} g_1(\alpha')
& =
\frac{1}{d} \sum_{p \text{ prime}} \sum_{\alpha' \in \cO(\alpha)} \log \max \left\{ |\alpha'|_p^{\omega - 1}, |\alpha'|_p^{\omega} \right\}.
\\ & \ge
\frac{\omega}{d} \sum_{p \text{ prime}}\sum_{\alpha' \in \cO(\alpha)} \log \max\{1,|\alpha'|_p\}.
\\ & =
 \frac{\omega }{d} \sum_{p \text{ prime}} - \log \left| a \right|_p.
\\ & =
 \frac{\omega }{d} \log |a|.
  \end{split}
\end{equation*}
Thus, the first inequality in \eqref{eq:2}  follows from the following consequence of Proposition~\ref{p:minima},
$$ \frac{1}{d} \sum_{\alpha' \in \cO(\alpha)} g_1(\alpha')
\ge
g_1(1)
= 12 \Fh(1). $$

The second inequality in \eqref{eq:2}  follows from a similar argument. Namely,

\begin{equation*}
  \begin{split}
12 \Fh(\alpha)
-
\frac{1}{d} \sum_{\alpha' \in \cO(\alpha)} g_1(\alpha')
& =
\frac{1}{d} \sum_{p \text{ prime}} \sum_{\alpha' \in \cO(\alpha)} \log \max \left\{ |\alpha'|_p^{\omega - 1}, |\alpha'|_p^{\omega} \right\}.
\\ & \ge
\frac{\omega-1}{d} \sum_{p \text{ prime}}\sum_{\alpha' \in \cO(\alpha)} \log \min\{1,|\alpha'|_p\}.
\\ & =
 \frac{\omega-1 }{d} \sum_{p \text{ prime}}  \log \left| b \right|_p.
\\ & =
 \frac{1-\omega }{d} \log |b|.
  \end{split}
\end{equation*}

\end{proof}

\begin{proof}[Proof of Corollary~\ref{c:numerical}]
The first inequality is a direct consequence of~\eqref{gmin} and Corollary~\ref{estimaciones especiales} and the second and the fifth from Theorem~\ref{t:minima}.
Furthermore, the fourth inequality follows from Corollary~\ref{estimaciones especiales}.

To prove the first statement and  the upper bound of~$\Fmuess$, for each~$a \ge 0.205$ we use Proposition~\ref{p:upper bound} with~$K$ equal to~$C_a \coloneqq \{ \zeta \in \C : |\zeta - a| = 1 \}$.
Lemmas~\ref{l:w uno} and~\ref{formula integral}, the estimate~\eqref{l:approx} and the formula~$g_{\D}(0) = - \log(1728 \pi^6) - \log|\Delta(\rho)|$ imply that for each~$a$ in~$(0, 2)$ we have
\begin{multline}
\label{punch}
\int_0^1 \ghyp \left(a + e^{2 \pi i t} \right) \dd t
\\ \le
g_{\D}(0)  - \frac{f'(0)}{13824}r_-(a)^3+6^3r_+(a)^6
- 6\int_0^1\log\left(1-r_+\left( \left|a + e^{2\pi i t} \right| \right)^2\right) \dd t.
\end{multline}
Taking~$a = 0.205$, note that the numbers $r_-(0.205)$ and~$r_+(0.205)$ can be computed to high precision.
Similarly, the function~$t \mapsto \log\left(1-r_+\left(|0.205+e^{2\pi i t}|\right)^2\right)$ is an explicit composition of sums, products, logarithms, sinus, cosinus and square roots, hence can be computed to high precision too (\emph{e.g.}, up to an error absolutely bounded by $10^{-15}$ in SAGE).
By Proposition~\ref{p:upper bound} with~$K = C_{0.205}$ and~\eqref{punch} with~$a = 0.205$, a numerical estimate gives
\begin{equation}
  \label{eq:12}
  \Fmuess
\le
\frac{1}{12} \int_0^1 \ghyp \left(0.205 + e^{2 \pi i t} \right) \dd t
\leq
-0.7486227509.
\end{equation}
The first assertion follows from this last estimate and from Proposition~\ref{p:upper bound}, by observing that the function
$$ a \mapsto \int_0^1 \ghyp \left(a + e^{2 \pi i t} \right) \dd t $$
is continuous and converges to~$\infty$ as~$a \to \infty$ by the asymptotic~\eqref{log log}.

On the other hand, the third inequality of the corollary follows from~\eqref{eq:12} and~\eqref{gmin}.

To prove the previous last statement of the corollary, let~$\alpha$ be an algebraic number that is not an algebraic integer and whose Faltings' height is less than or equal to~$\Fmuess$.
Then~\eqref{eq:12}, Corollary~\ref{estimaciones especiales}, Proposition~\ref{l:g uno} and a numerical estimate imply that the right-hand side of the first inequality in ~\eqref{eq:2} with~$\Fh(\alpha)$ replaced by~$\Fmuess$ is less than or equal to~$1/ 15177$. 
 On the other hand, our assumption that~$\alpha$ is not an algebraic integer implies that the number~$a$ as in the statement of Lemma~\ref{l:almost integer} satisfies $|a| \ge 2$.
Thus, by~\eqref{eq:2} and our hypothesis~$\Fh(\alpha) \le \Fmuess$, the degree~$d$ of~$\alpha$ satisfies $d \ge \log 2 \cdot 15177 > 10519$.

Now assume that $\alpha$ is an algebraic number of degree at most $10$ such that $h_F(\alpha) \leq \Fmuess$. By the previous considerations, $\alpha$ is an algebraic integer. Then, using the second inequality in ~\eqref{eq:2},  combined with the estimates~\eqref{eq:12}, Corollary~\ref{estimaciones especiales}, Proposition~\ref{l:g uno} and a numerical estimate, the parameter $b$ in the statement of Lemma~\ref{l:almost integer} can be bounded from above as  
$$|b| \leq \exp \left(10\cdot 1032\cdot12\cdot(\Fmuess-h_F(1)) \right) \leq 1.98.$$ 

Since $|b|$ is a positive integer, we conclude that $|b|=1$ and that $\alpha$ is an  algebraic unit.  This completes the proof of the last statement and of the corollary.
\end{proof}

\begin{proof}[Proof of Theorem~\ref{t:roots of unity}]
The hypotheses on~$n$ imply either~$\mu(n) \in \{0, 1 \}$ or~$\phi(n) \ge 12$.
In all the cases, $\frac{\mu(n)}{\phi(n)} \ge - \frac{1}{12}$.
Using Corollary~\ref{estimaciones especiales}, a numerical estimate gives
$$ \Fh(\zeta_n)
\ge
- 0.748622711. $$
Then the theorem follows from~\eqref{eq:12}.
\end{proof}

\section{Proof of Proposition~\ref{p:minima}}\label{key}
First, we establish the following numerical estimate of~$\partial_x \ghyp(1)$ implying the inequalities $0 < \partial_x \ghyp(1) < 1$ (Proposition~\ref{l:g uno}).
These estimates are also used below to show convexity properties of~$\ghyp$.
The proof of the remaining part of Proposition~\ref{p:minima} is divided in three cases, according to the proximity to the unit disk.

Throughout the rest of this section we use the functions~$\whf$, $f$ and~$g_{\D}$, defined in~\eqref{fgorro}, \eqref{def f} and~\eqref{g en el disco}, respectively.

\begin{prop}
\label{l:g uno}
We have $\frac{1}{1032} \le \partial_x \ghyp(1) \le \frac{1}{1025}.$
\end{prop}

\begin{proof}
Let~$r_1$ be the only number in~$(0,r_0)$ such that $\whf(r_1) = 1$.
Since~$\whf$ is real (Lemma~\ref{conjugaciones}), $f$ is also real.
Together with the fact that~$f$ is univalent and that~$f'(0) > 0$ (Lemma~\ref{radio}), we conclude that~$\whf'(r_1) = 3r_1^2f'(r_1^3) > 0$.

On the other hand, \eqref{invarianza} implies that $\partial_y \ghyp(1)=0$.
Hence, 
\begin{equation}
\label{identidad basica}
\partial_x \ghyp(1)
=
2 \Re \left( \partial \ghyp(1) \right)
=
\frac{2 \Re \left( \partial g_\D(r_1) \right)}{{\whf'}(r_1)}.
\end{equation}
Using~\eqref{sobre gD} and~\eqref{quinto orden}, we obtain 

\begin{eqnarray}
2\Re \left(\partial g_\D(r_1)\right)
&=&
\frac{12r_1}{1-r_1^2} - \Re \left( (\log \whh )' (r_1) \right) \nonumber \\
&=&
\frac{12r_1}{1-r_1^2} -3\cdot \frac{f'(0)}{13824} r_1^2 + 4 \cdot E,\nonumber 
\end{eqnarray}
where $|E|\leq 6^4r_1^5$.
Moreover, Lemma~\ref{l:w uno} ensures that $r_1^-\leq r_1 \leq r_1^+$, leading to

\begin{eqnarray}
2 \Re \left( \partial g_\D(r_1) \right)
&\leq &\frac{12r_1^+}{1-(r_1^+)^2} -3\cdot \frac{f'(0)}{13824} (r_1^-)^2 +4 \cdot 6^4(r_1^+)^5 \nonumber \\
2 \Re \left( \partial g_\D(r_1) \right)
& \geq &  \frac{12r_1^-}{1-(r_1^-)^2} -3\cdot \frac{f'(0)}{13824} (r_1^+)^2 -4 \cdot 6^4(r_1^+)^5.\nonumber  
\end{eqnarray}

\noindent These estimates, together with~\eqref{identidad basica} and Lemma~\ref{zetaw}, \ref{le:j'}, prove the claim.
\end{proof}

To complete the proof of Proposition~\ref{p:minima}, let~$T$ be defined by~\eqref{dominio fundamental}, fix~$\zeta$ in~$\C$, and let~$\tau$ in~$T$ be such that~$j(\tau) = \zeta$.
There are three cases, according to the location of~$\tau$.

We also use the following estimate several times:
\begin{equation}
  \label{eq:14}
  \ghyp(1) \le - 8.89835372,
\end{equation}
which is a direct consequence of Corollary~\ref{estimaciones especiales}, and the formula~$\ghyp(1) = 12 \Fh(1)$, \emph{cf.}~\eqref{e:global/local}.

\subsubsection*{Case 1. $\Im(\tau) \ge 1$.}
\begin{lemm}
\label{l:g cusp approximation}
For~$\tau$ in~$\H$ satisfying~$\Im(\tau) \ge 1$, we have
$$ \left| \gH(\tau) - \left( 2\pi \Im(\tau) - 6 \log(\Im(\tau)) - 6
    \log(4\pi) \right) \right|
\le
\frac{24}{\exp(2\pi) - 2}. $$
\end{lemm}

\begin{proof} Let~$\tau$ in~$\H$ be such that~$\Im(\tau) \ge 1$, and note that~$q =
e^{2 \pi i \tau}$ satisfies~$|q| \le \exp(-2\pi)$.
This implies that for every integer~$r \ge 1$ we have
$$ |\log(1 - q^r)|
\le
\frac{\exp(2\pi)}{\exp(2\pi) - 1} |q|^r
\text{ and }
\left| \sum_{r = 1}^{\infty} \log(1 - q^r) \right|
\le
\frac{1}{\exp(2\pi) - 2}. $$
Then the desired estimate is obtained by applying the definition of~$\gH$.
\end{proof}

\begin{lemm}
\label{l:j cusp}
For every~$\tau$ in~$\H$ satisfying~$\Im(\tau) \ge 1$, we have
$$ |j(\tau)| \le 4 \exp(2\pi \Im(\tau)). $$
\end{lemm}

\begin{proof} Let~$\whj \colon \D \to \CC$ be the holomorphic function such that~$\whj(0)
= \infty$, and such that for every~$\tau$ in~$\H$ we have~$\whj(q(\tau)) = j(\tau)$.
Since this function is univalent on~$B(0, \exp(-2\pi))$, and the
derivative of~$\whj^{-1}$ at~$q = 0$ is equal to~$1$, by the
Koebe one quarter theorem \cite[Corollary~1.4, p.~22]{Pommerenke75} for every~$\tau$ in~$\H$ satisfying~$\Im(\tau) \ge 1$, we
have
$$ |j(\tau)|^{-1}
=
|\whj(q(\tau))|^{-1}
\ge
\frac{1}{4} |q(\tau)|
=
\frac{1}{4}\exp(-2\pi \Im(\tau)).$$
\end{proof}

We now proceed to the proof of Proposition~\ref{p:minima} in the case where~$\tau$ satisfies $\Im(\tau) \ge 1$.
First note that by Proposition~\ref{l:g uno}, \eqref{eq:14} and Lemma~\ref{l:g cusp approximation}, we have
\begin{equation*}
  \begin{split}
\gH(\tau) - \ghyp(1)
& \ge
2\pi \Im(\tau) - 6 \log(\Im (\tau)) - 6 \log(4\pi) - \ghyp(1)
- \frac{24}{\exp(2\pi) - 2}
\\ & \ge
2\pi \Im(\tau) - 6 \log(\Im (\tau)) - 6.25
\\ & \ge
0.03 \Im(\tau)
\\ & \ge
0.02 + \partial_x \ghyp(1)(10 \Im(\tau))
\\ & \ge
0.02 + \partial_x \ghyp(1)(2\pi \Im(\tau) + \log 4).
  \end{split}
\end{equation*}
Combined with Lemma~\ref{l:j cusp} and the definition of~$g_1$, this
implies
$$ g_1(\zeta) = g_1(j(\tau)) \ge g_1(1) + 0.02. $$
This proves Proposition~\ref{p:minima} in the case where~$\Im(\tau)
\ge 1$.

\subsubsection*{Case 2. $\frac{1}{\pi} \log(19) \le \Im (\tau) \le 1$.}
\begin{lemm}\label{cota jota}
For each~$\tau$
in~$T$ satisfying~$\ \Im(\tau) \le 1$, we have $|j(\tau)|\leq 1728$. 
\end{lemm}

\begin{proof} Note that the image of~$\{ \tau \in T : \Im(\tau) \le 1\}$ is a Jordan domain bounded by the curve~$j(\{ \tau \in T : \Im(\tau) = 1 \})$.
So, it is enough to prove the inequality in the case~$\Im(\tau) = 1$.
Using that the coefficients in the $q$-expansion of~$j$ are positive,
for every~$x$ in~$\R$ we have~$|j(x + i)| \le j(i) = 1728$, finishing the proof of the lemma.
\end{proof}

\begin{lemm}\label{primer termino}
 For each~$\tau$
in~$T$ satisfying~$\frac{1}{\pi} \log(19) \le \Im(\tau) \le 1$, we have
$$\gH(\tau) \geq 2 \log(19) - 6 \log (4 \log(19))
- 24 \log \left( \frac{19^2 + 1}{19^2} \right).$$
 \end{lemm}

\begin{proof} Lemma~\ref{l:really increasing}, combined with~\eqref{eq:6}, imply that for each~$\tau$
in~$T$ satisfying~$\frac{1}{\pi} \log(19) \le \Im(\tau) \le 1$, we have
\begin{equation*}
\begin{split}
\gH(\tau)
& \ge
\gH \left( \frac{1}{2} + i \frac{1}{\pi} \log(19) \right)
\\ & =
2 \log(19) - 6 \log (4 \log(19))
- 24 \sum_{n = 1}^{\infty} \log \left( 1 - \left(- \frac{1}{19^2} \right)^n \right).
  \end{split}
\end{equation*}
Then, we are reduced to show that  
\begin{equation}\label{negativo}
 \sum_{n =2}^{\infty} \log \left( 1 - \left(- \frac{1}{19^2} \right)^n \right) \leq 0.
\end{equation}

Setting $s=\frac{1}{19^2}$, the arithmetic-geometric mean implies that for each $r\geq 1$,
$$ \left( 1-(-s)^{2r} \right) \left( 1 - (-s)^{2r+1} \right)
=
\left( 1-s^{2r} \right) \left( 1+s^{2r+1} \right)
\leq
\left( 1-\frac{s^{2r}-s^{2r+1}}{2}\right)^2.$$ Since $0<s<1$, the last quantity is strictly less than 1. Hence,
$$ \prod_{r = 1}^\infty \left( 1 - (-s)^{2r} \right) \left( 1 - (-s)^{2r+1} \right)
<
1,$$
justifying~\eqref{negativo}.
\end{proof}

We now proceed to the proof of Proposition~\ref{p:minima} in the case where $\tau$ satisfies $\frac{1}{\pi} \log(19) \le \Im(\tau) \le 1$.
Proposition~\ref{l:g uno}, \eqref{eq:14} and Lemmas~\ref{cota jota} and~\ref{primer termino} imply that
$$ g_1(\zeta) - g_1(1)
=
g_1 \circ j(\tau) - \ghyp(1)
=
\gH(\tau) - \partial_x \ghyp(1) \log|j(\tau)| - \ghyp(1) $$
is bounded from below by 
$$2 \log(19) - 6 \log (4 \log(19))
- 24 \log \left( \frac{19^2 + 1}{19^2} \right)
- \frac{1}{1025}\log 1728
+ 8.9835372
\geq
10^{-3}, $$
finishing the proof of Proposition~\ref{p:minima} in this  case.

\subsubsection*{Case 3. $\Im (\tau) \leq \frac{1}{\pi} \log(19)$.}

\begin{lemm}\label{tamano}
Let~$\psi \colon \D \to \H$ be as defined in~\eqref{def psi}.
Then for every~$\tau$ in~$T$ satisfying~$\Im(\tau) \le  \frac{1}{\pi} \log(19)$, we
have
$$ \left| \psi^{-1}(\tau) \right| \le 1 - \frac{\pi}{2 \sqrt{3}}. $$
\end{lemm}

\begin{proof} Put~$I \coloneqq  \frac{1}{\pi} \log(19)$.
Since the image by~$j$ of the set~$\{ \tau \in T : \Im (\tau) = I
\}$ is a Jordan curve, it is enough to prove the lemma in the case
where~$\Im(\tau) = I$.
By symmetry, it is enough to consider the case where~$\Re(\tau)
\ge 0$.

Put
$$ \tau_2 = \sqrt{1 - I^2} + i I
\text{ and }
\tau_2' = \frac{1}{2} + i I. $$
Note that the image by~$\psi^{-1}$ of the line~$\{ \tau \in \H :
\Im(\tau) = I \}$ is a circle that is tangent to the unit circle at~$w = -1$, and that is
contained in the left half plane.
Thus, the image by~$\psi^{-1}$ of the segment~$[\tau_2, \tau_2']$ is the
arc of this circle that is contained in the angular sector bounded by
the rays~$\{ t < 0 \}$ and~$\{ - t \rho : t > 0 \}$.
It follows that for each~$\tau$ in the segment~$[\tau_2, \tau_2']$,
we have
$$ |\psi^{-1}(\tau)|
\le
|\psi^{-1}(\tau_2)|
=
\frac{1 - 2\sqrt{1 - I^2}}{2 + \sqrt{3}I - \sqrt{1 - I^2}}
\le
1 - \frac{\pi}{2 \sqrt{3}}.$$
\end{proof}

\begin{lemm}
\label{l:real minimization}
For every~$w$ in~$\D$ such that~$0 < |w| \le 1 - \frac{\pi}{2 \sqrt{3}}$, we have~$g_1 \circ \whf(w) \ge g_1 \circ \whf(|w|)$, with equality if and only if~$w^3 = |w|^3$.
\end{lemm}

The proof of this lemma is given after the following lemma.

\begin{lemm}\label{erres grandes} 
Let~$J \colon B(0, r_0) \setminus \{ 0 \} \to \C$ be defined by~$J(w)\coloneqq\left( \frac{\whf'}{\whf} \right) (w)\cdot w$.
Then for every~$w$ in~$B\left(0, 1 - \frac{\pi}{2 \sqrt{3}} \right) \setminus \{ 0 \}$ we have $|J'(w)| \leq 4000 |w|^2$.
 \end{lemm}

\begin{proof}
Using $J(w) = 3w^3 \left(\frac{f'}{f} \right)(w^3)$ and applying the first two inequalities in Lemma~\ref{Mas Koebes} with $z = w^3/r_0^3$,  we obtain
 $$|J(w)|\leq 3 \frac{1 + |z|}{1 - |z|}
\quad \text{ and } \quad
\left| \frac{J(w)}{3} - 1 \right| \leq |w|^3 \left( \frac{2}{r_0^{3}} \cdot \frac{(1 + |z|)^2}{(1 - |z|)^3} \right). $$
Since these upper bounds are increasing in~$|z|$, a numerical estimate with~$|z|$ replaced by~$\left( 1 - \frac{\pi}{2 \sqrt{3}} \right)^3/r_0^3$ gives
$$ |J(w)| \le 4
\text{ and }
|J(w)-3|\leq 400 |w|^3. $$
Using these inequalities and the third inequality in Lemma~\ref{Mas Koebes}, we have
$$ |J'(w) \cdot w|
=
\left| J(w) \left( 3-J(w) \right) - 9w^6 \left( \frac{f''}{f} \right) (w^3) \right|
\le
|w|^3 \left( 1600 + \frac{18}{r_0^{3}} \cdot \frac{2 + |z|}{(1 - |z|)^2} \right).$$
The desired inequality follows by observing that the upper bound is increasing in~$|z|$ and by estimating it with~$|z|$ replaced by~$\left( 1 - \frac{\pi}{2 \sqrt{3}} \right)^3 / r_0^3$.
\end{proof}

\begin{proof}[Proof of Lemma~\ref{l:real minimization}]
By Theorem~\ref{Koebe Thm}, \ref{unoKoebe}, applied to~$f_0$ defined in~\eqref{eq:11} and~$z = \left(  \frac{w}{r_0} \right)^3$, we have
$$ \log |\whf(w)| - \log |\whf(|w|)|
=
\log \left( \frac{|f(w^3)|}{f(|w|^3)} \right)
\le
2 \log \left(\frac{1 + r_0^{-3} |w|^3}{1 - r_0^{-3} |w|^3}\right)
\le 6r_0^{-3} |w|^3. $$
Here, we have used the elementary inequality
$$ \quad 0\leq x \leq \left(\frac{1-\frac{\pi}{2\sqrt{3}} }{2-\sqrt{3}}\right)^3  \Rightarrow \log\left(\frac{1+x}{1-x}\right) \leq 3x.$$

Assume first that~$w$ satisfies $\Re \left( \frac{w^3}{|w|^3} \right) \le \frac{1}{2}$.
Then, by Lemma~\ref{radio}, Proposition~\ref{l:g uno}, and Proposition~\ref{p:approximation}, we have

\begin{eqnarray}
g_1 \circ \whf (w) - g_1 \circ \whf (|w|)
&\ge&
\frac{f'(0)}{13824} |w|^3 \left( 1 - \Re \left( \frac{w^3}{|w|^3} \right) \right)
- 2 \cdot 6^3 |w|^6 - \left(\frac{6}{1025} r_0^{-3}\right) |w|^3\nonumber \\
&\ge&  |w|^3\left( \frac{237698}{13824}\cdot \frac{1}{2} -2\cdot 6^3 \cdot \left(1-\frac{\pi}{2\sqrt{3}}\right)^3 -\frac{6 r_0^{-3}}{1025} \right)\nonumber\\
&\ge& |w|^3\nonumber\\
&> &0.\nonumber
\end{eqnarray}

Now we assume $ \Re \left( \frac{w^3}{|w|^3} \right) > \frac{1}{2}$, put~$r \coloneqq |w|$ and~$\theta \coloneqq \arg(w)$, and let~$H \colon \R \to \R$ be the function
defined by~$H(\ttheta) \coloneqq g_1 \circ \whf \left( r \exp \left( i \ttheta \right) \right)$.
Using the function~$J$ defined in Lemma~\ref{erres grandes}, we have
\begin{equation*}
  \begin{split}
H''(\theta)
& =
\Re \left[ (\log \whh)''(r \exp(i\theta)) r^2 \exp(2i\theta)
+ (\log \whh)' (r \exp(i\theta)) r \exp(i\theta)
\right. \\ & \quad \left.
+ \partial_x \ghyp(1) \cdot  J'(r \exp(i\theta))\cdot r \exp(i\theta)
 \right].
  \end{split}
\end{equation*}
Combining Proposition~\ref{l:g uno}, Proposition~\ref{p:approximation}, \eqref{quinto orden} and~\eqref{cuarto orden} and Lemmas~\ref{radio} and~\ref{erres grandes}, and using $\Re(\exp(3i\theta)) \ge\frac{1}{2}$, we have
\begin{equation*}
  \begin{split}
H''(\theta)
 & \ge 
\frac{9 f'(0)}{13824} r^3 \Re(\exp(3i\theta))
-5 \cdot 6^4 r^6 - 6^4 r^6 - 4 r^3
 \\ & \geq
r^3 \left( \frac{9 \cdot 237698}{13824} \cdot \frac{1}{2} -6^5 \left(1 - \frac{\pi}{2 \sqrt{3}} \right)^3 - 4 \right)
 \\ & \ge r^3.
  \end{split}
\end{equation*}

This proves that, if we denote by~$\theta_0$ the unique number in~$\left[0,
  \frac{\pi}{3} \right]$ such that~$\Re(\exp(3i\theta_0)) =
\frac{1}{2}$, then~$H$ is strictly convex on $[- \theta_0,
\theta_0]$.
Moreover, Lemma~\ref{conjugaciones} implies that~$H$ is even, hence it attains its minimum
on~$[-\theta_0, \theta_0]$ at, and only, at~$\theta = 0$.
This completes the proof of the lemma.
\end{proof}

\begin{lemm}
\label{l:convexity}
The restriction~$V$ of~$g_1 \circ \whf$ to~$\left(0, 1 - \frac{\pi}{2 \sqrt{3}} \right]$ is strictly convex.
Moreover, if~$r_1$ is the only number in $\left(0, 1 - \frac{\pi}{2 \sqrt{3}} \right)$ such that $\whf(r_1)=1$, \emph{cf}. Figure~\ref{mono}, then~$V$  attains its minimum at, and only at, $r=r_1$.
\end{lemm}

\begin{proof}
Since~$\whf$ is real (Lemma~\ref{conjugaciones}), $f$ is also real.
Together with the fact that~$f$ is univalent and that~$f(0) = 0$ and~$f'(0) > 0$ (Lemma~\ref{radio}), we conclude for each~$r$ in~$(0, r_0)$ we have~$J(r) = 3r^3f'(r^3)/f(r^3) > 0$.

By~\eqref{cuarto orden} in Proposition~\ref{p:approximation}, Proposition~\ref{l:g uno}, and Lemmas~\ref{radio} and~\ref{erres grandes}, we have
\begin{equation*}
  \begin{split}
V''(r)
& =
 \frac{12(1+r^2)}{(1 - r^2)^2} - (\log \whh)''(r) +   \partial_x \ghyp(1) \Re\left(\frac{J(r)}{r^2} \right) -  \partial_x \ghyp(1) \Re \left( \frac{J'(r)}{r} \right)   
\\  & \ge 
 \frac{12(1+r^2)}{(1 - r^2)^2}  -\frac{6f'(0)}{13824}r -5\cdot 6^4\cdot r^4
-\frac{1}{1025} 4000 r
 \\ & \ge
12 - \left(1 - \frac{\pi}{2 \sqrt{3}} \right) \left( \frac{6 \cdot 237698}{13824} + 5\cdot 6^4\cdot \left(1 - \frac{\pi}{2 \sqrt{3}} \right)^3 + 4 \right)
\\ & \ge
1.
  \end{split}
\end{equation*}
This proves that~$V$ is strictly convex on~$\left(0,1-\frac{\pi}{2\sqrt{3}} \right)$.

To finish the proof, it is enough to show that $r=r_1$ is a critical point of~$V$.
Indeed,
$$ V'(r)
=
\partial_x g_\D (r) - \partial_x \ghyp(1)\Re\left(\left( \frac{\whf'}{\whf} \right) (r) \right).$$
The relation $g_\D=\ghyp \circ \whf$ implies $\partial_x g_\D = (\partial_x \ghyp)\circ \whf \cdot \partial_x \whf$.
Since $\whf(r_1)=1$ and~$\whf$ is real, we see that $V'(r_1)=0$.
This completes the proof of the lemma.
\end{proof}

We now proceed to the proof of Proposition~\ref{p:minima} in the remaining case where $\tau$ satisfies $\Im(\tau) \leq \frac{1}{\pi}\log(19) $.
Lemma~\ref{tamano} ensures that~$w \coloneqq \psi^{-1}(\tau)$ satisfies $|w|\leq 1-\frac{\pi}{2\sqrt{3}}$.  Then, combining Lemmas~\ref{l:real minimization} and~\ref{l:convexity}, we have
$$ g_1(\zeta)
=
g_1 \circ j(\tau)
=
g_1 \circ \whf (w)
\ge
g_1 \circ \whf (|w|)
\geq
g_1 \circ \whf(r_1)
=
g_1(1),$$
with equality if and only if $j(\tau)=1$. This finishes the proof of Proposition~\ref{p:minima}.

\section{Numerical experiments}
\label{sec:numer-exper}

In this section we briefly describe our numerical experiments that give us two more minima of the stable Faltings height, both of which are larger than~$\Fh(0)$ and~$\Fh(1)$.
We use the procedure described Section~\ref{modular} to find a lower bound of~$\Fmuess$, with a carefully chosen family of sections.
The minima of~$\Fh$ that we find are attained at the common support of these sections.
See~\cite{BMRan} for the SAGE source code we use in our experiments and further details.

Recall the metrized line bundle $\LLL=(M_{12},\npet{\cdot})$ of weight 12
modular forms with the Petersson metric defined in Section~\ref{modular}. We have that $\LL \simeq
O_{\X}(D_\infty)$. 
The sections of $\LL^{\otimes n}$ are in one to one correspondence with the
space of homogeneous polynomials of degree $n$ with integral
coefficients in the variables $X,Y$, where $[X:Y]$ are
homogeneous coordinates of $\p$ and the point at infinity is~$[1:0]$. 

We start with the section $s_{0}= \Delta $.
Using the notation introduced above, this is the section~$Y$ that has a zero at infinity. 
Then $s_{0}^{1/12}$ has weight one and $g_{s_{0}^{1/12}}(\zeta)=\frac{1}{12}\ghyp(\zeta)$, so by Lemma~\ref{l:really increasing} we have
\begin{displaymath}
  \inf_{\zeta \in \X(\C)}g_{s_{0}^{1/12}}(\zeta)
=
g_{s_{0}^{1/12}}(0)
=
-0.74875248\dots,
\end{displaymath}
which proves that the minimum value of Faltings' height is~$h_0 \coloneqq \Fh(0) = g_{s_{0}^{1/12}}(0)$.

We next define
$$ s_{1} \coloneqq X,
a_{1,1} \coloneqq \frac{1}{12}\partial_{x}\ghyp(1)
\text{ and }
s_{a_{1,1}}
\coloneqq
s_{1}^{a_{1,1}}s_{0}^{1/12-a_{1,1}}
=
X^{a_{1,1}} Y^{1/12-a_{1,1}}, $$
so that~$12 g_{s_{a_{1, 1}}} = g_1$ is the function defined in Proposition~\ref{p:minima}.
By this proposition we know that $g_{s_{a_{1,1}}}$ attains its
minimum at the point $1$, so
\begin{displaymath}
  h_{1}\coloneqq
  \Fh(1)=g_{s_{a_{1,1}}}(1)
=
\inf_{\zeta \in \X(\ce)}g_{s_{a_{1,1}}}(\zeta)
=
-0.74862817\dots
\end{displaymath}
To check that this value the second minimum of~$\Fh$, write
\begin{displaymath}
  s_{2}
\coloneqq
X - Y
\end{displaymath}
and consider sections of the form
\begin{displaymath}
  s_{a_{2,1},a_{2,2}}
\coloneqq
s_{1}^{a_{2,1}}s_{2}^{a_{2,2}}s_{0}^{1/12-a_{2,1}-a_{2,2}}.
\end{displaymath}
We compute numerically
\begin{displaymath}
\sup_{a_{2,1},a_{2,2}}\inf _{\zeta \in \X(\ce)} g_{s_{a_{2,1},a_{2,2}}} (\zeta)
=
-0.74862517\dots,
\end{displaymath}
which is a lower bound of~$\Fh$ on~$\overline{\Q} \setminus \{0, 1 \}$.
Since this number is larger than~$h_1$, this proves that $h_{1}$ is the second minimum of~$\Fh$ on~$\overline{\Q}$.
The experimental values of the coefficients are
\begin{displaymath}
  a_{2,1}=0.0000808846,\qquad a_{2,2}=0.000006017184
\end{displaymath}
and the new minimum is attained at the points
\begin{displaymath}
0.50004865+i∗0.86601467
\text{ and }
0.50004865+i∗0.86601467.
\end{displaymath}

The numbers above are very close to the solutions~$\rho$ and~$\overline{\rho}$ of the equation
$z^{2}-z+1=0$, which are the primitive roots of unity of order~$6$.
We write
\begin{displaymath}
  h_{2}\coloneqq \Fh(\rho)=-0.74862517\dots
\text{ and }
s_{3} \coloneqq X^{2} - XY + Y^{2}
\end{displaymath}
and consider sections of the form
\begin{displaymath}
  s_{a_{3,1},a_{3,2},a_{3,3}}
\coloneqq
s_{1}^{a_{3,1}}s_{2}^{a_{3,2}}s_{3}^{a_{3,3}}s_{0}^{1/12-a_{3,1}-a_{3,2}-2a_{3,3}}.   
\end{displaymath}
Numerically we obtain
\begin{displaymath}
  \sup_{a_{3,1},a_{3,2},a_{3,3}}\inf _{\zeta \in \X(\ce)} g_{s_{a_{3,1},a_{3,2},a_{3,3}}} (\zeta)
=
-0.74862386\dots,
\end{displaymath}
which is a lower bound of~$\Fh$ on~$\overline{\Q} \setminus \{0, 1, \rho, \overline{\rho} \}$.
Since this number is larger than~$h_2$, this shows that $h_{2}$ is the third minimum of~$\Fh$ on~$\overline{\Q}$.
The coefficients we obtain are
\begin{displaymath}
  a_{3,1}=0.00007979626,\quad  a_{3,2}=0.000004433084,\quad
  a_{3,3}=0.000002454098.
\end{displaymath}

Testing other roots of unity, we found that if~$\xi $ is a primitive root of unity of order~$10$, then
\begin{displaymath}
  h_{3}\coloneqq \Fh(\xi)=-0.74862366 \dots
\end{displaymath}
is close to the next possible value of Faltings' height.
The corresponding cyclotomic polynomial is~$z^4 - z^3 + z^2 - z + 1$, so we write
$$ s_{4} \coloneqq X^{4} - X^{3}Y + X^{2}Y^{2} - XY^{3} + Y^{4} $$
and consider sections of the form
\begin{displaymath}
  s_{a_{4,1},a_{4,2},a_{4,3},a_{4,4}}
\coloneqq
s_{1}^{a_{4,1}}s_{2}^{a_{4,2}} s_{3}^{a_{4,3}}s_{4}^{a_{4,4}}s_{0}^{1/12-a_{4,1}-a_{4,2}-2a_{4,3}-4a_{4,4}}.
\end{displaymath}
Numerically we obtain
\begin{displaymath}
\sup_{a_{4,1},a_{4,2},a_{4,3},a_{4,4}}\inf _{\zeta \in \X(\ce)} g_{s_{a_{4,1},a_{4,2},a_{4,3},a_{4,4}}}
=
-0.74862360\dots 
\end{displaymath}
which gives us a new lower bound for~$\Fmuess$ and shows that $h_{3}$ is the fourth minimum of Faltings' height.
The corresponding coefficients are
\begin{alignat*}{2}
  a_{4,1}&=0.000078055985,\qquad& a_{4,2}&=0.000003803298,\\  a_{4,3}&=0.000002385096,\qquad&  a_{4,4}&=0.000000865203.
\end{alignat*}

\bibliographystyle{alpha}

\Addresses

\end{document}